\documentclass[reqno]{amsart}

\usepackage[
left=1.25in, right=1.25in]{geometry}
\usepackage[all]{xy}
\usepackage{graphicx}
\usepackage{indentfirst}
\usepackage{bm}
\usepackage{amsthm}
\usepackage{newtxtext}
\usepackage{mathrsfs} 
\usepackage{latexsym}
\usepackage{amsmath}
\usepackage{amssymb}
\usepackage{color,soul}
\usepackage{hyperref}
\usepackage{microtype}
\usepackage{tikz}
\usetikzlibrary{arrows}
\usepackage{comment}
\usepackage{color}
\usetikzlibrary{positioning,calc}
\usetikzlibrary{shapes.geometric, arrows}
\usetikzlibrary{decorations.markings}
\usetikzlibrary{decorations.pathmorphing}
\usetikzlibrary{decorations.text}
\usetikzlibrary{decorations.shapes}
\usetikzlibrary{decorations.fractals}
\usetikzlibrary{decorations.footprints}
\usepackage[title]{appendix}

\begin{document}
\newtheorem*{remark*}{Remark}
\newtheorem{theorem}{Theorem}[section]
\newtheorem{proposition}[theorem]{Proposition}
\newtheorem{corollary}[theorem]{Corollary}
\newtheorem{definition}[theorem]{Definition}
\newtheorem{lemma}[theorem]{Lemma}
\newtheorem{example}[theorem]{Example}
\newtheorem{exercise}[theorem]{Exercise}
\newtheorem{remark}[theorem]{Remark}
\newtheorem{question}[theorem]{Question}
\newtheorem{conjecture}[theorem]{Conjecture}
\newtheorem{claim}[theorem]{Claim}
\newtheorem{fact}[theorem]{Fact}
\newtheorem*{ac}{Acknowledgements}
\newtheorem{notation}[theorem]{Notation}

\newcommand{\jinsong}[1]{\textbf{\textcolor[rgb]{0.00,0.00,1.00}{[[#1 -Jinsong]]}}}
\newcommand{\zhengwei}[1]{\textbf{\textcolor[rgb]{0.00,1.00,1.00}{[[#1 -Zhengwei]]}}}

\newcommand{\Projharm}{\mathcal{P}_{H_2}}
\newcommand{\A}{\mathscr{A}}
\newcommand{\M}{\mathcal{M}}
\newcommand{\NH}{\text{dim Hom}}
\newcommand{\RG}{\mathcal{R}(G)}
\newcommand{\ssl}{\mathfrak{sl}}

\newcommand{\Projep}{\mathcal{P}_{F\!N}}
\newcommand{\C}{\mathscr{C}}
\newcommand{\D}{\mathscr{D}}
\newcommand{\El}{\mathscr{E}_{\pi}} 
\newcommand{\Hl}{\mathscr{R}_{\pi}} 
\newcommand{\Bl}{B_{\pi}} 
\newcommand{\Hlz}{\mathscr{R}_{\pi, 0}} 
\newcommand{\CR}{\mathcal{C\!R}}
\newcommand{\SB}{\mathcal{S}_{B}} 
\newcommand{\SBz}{\mathcal{S}_{B,0}} 
\newcommand{\SPi}{\mathcal{S}_{\pi}} 
\newcommand{\g}{\mathfrak{g}}
\newcommand{\Z}{\mathbb{Z}}
\newcommand{\fw}{\mathbf{W}_f}
\newcommand{\sr}{\mathbf{R}_s}
\newcommand{\sW}{\mathscr{W}}
\newcommand{\sHW}{\mathscr{H\!W}}
\newcommand{\sWl}{\mathscr{W}_{\ell}}
\newcommand{\sR}{\mathscr{R}}
\newcommand{\sHR}{\mathscr{H\!R}}
\newcommand{\sRa}{\mathscr{R}_{a}}
\newcommand{\sRl}{\mathscr{R}_{\ell}}
\newcommand{\V}{\mathcal{V}}
\newcommand{\m}{\mathbf{m}}
\newcommand{\kk}{\mathbf{k}}
\newcommand{\jj}{\mathbf{j}}
\newcommand{\vv}{\mathbf{v}}
\newcommand{\ssb}{\mathbf{s}}
\newcommand{\ee}{\mathbf{e}}
\newcommand{\rr}{\mathbf{r}}
\newcommand{\ww}{\mathbf{w}}
\newcommand{\G}{G}
\newcommand{\N}{\mathcal{N}}
\newcommand{\bR}{\mathbf{R}}
\newcommand{\CL}{\mathcal{C}_\ell}
\newcommand{\wchi}{\widetilde{\chi}}
\newcommand{\hchi}{\widehat{\chi}}
\newcommand{\rk}{n_r} 
\newcommand{\spec}{\mathbf{Spec}_{\ell}} 
\newcommand{\rooth}{\mathscr{R}_h} 
\newcommand{\hr}{\rr} 
\newcommand{\hil}{\mathcal{H}} 
\newcommand{\tg}{\tilde{g}} 
\newcommand{\tf}{\tilde{f}} 
\newcommand{\ttj}{\tilde{\vartheta_{\jj}}} 
\newcommand{\ttr}{\tilde{\vartheta_{\hr}}} 

\title{Antisymmetric characters and Fourier duality } 
\author{Zhengwei Liu}
\address{Harvard University}
\email{zhengweiliu@fas.harvard.edu}
\author{Jinsong Wu}
\address{IASM, Harbin Institute of Technology and Harvard University}
\email{wjs@hit.edu.cn}
\date{}
\maketitle

\begin{abstract}
Inspired by the quantum McKay correspondence, we consider the classical $ADE$ Lie theory as a quantum theory over $\ssl_2$.
We introduce anti-symmetric characters for representations of quantum groups and investigate the Fourier duality to study the spectral theory. In the $ADE$ Lie theory, there is a correspondence between the eigenvalues of the Coxeter element and the eigenvalues of the adjacency matrix.  We formalize related notions and prove such a correspondence for representations of Verlinde algebras of quantum groups: this includes the quiver of any module category acted on by the representation category of any simple Lie algebra $\g$ at any level $\ell$. This answers an old question posed by Victor Kac in 1994 and a recent comment by Terry Gannon. 
\end{abstract}


\section{Introduction}
\color{black}
In 1980, McKay found his correspondence between subgroups of $SU(2)$ and the affine $ADE$ Dynkin diagrams \cite{Mck80}. In 1987, Cappelli, Itzykson and Zuber found a correspondence between $ADE$ Dynkin diagrams and the modular invariance of quantum $\ssl_2$ \cite{CIZ87}, which was further formulated by Ocneanu as a correspondence to subgroups of quantum $\ssl_2$ \cite{Ocn02}. This has been considered as a quantum analogue of the McKay Correspondence. We elaborate this idea in \S \ref{Sec:Backgroupnd}.

In $ADE$ Lie theory, the action of the Coxeter element on the root system has periodicity equal to the Coxeter number $n_c$. Its eigenvalue is $e^{\frac{k 2\pi i}{n_c}}$, for some $k\in \Z_{n_c}$, and $k$ is called a Coxeter exponent.  
On the other hand, the eigenvalues of the adjacency matrix of the $ADE$ Dynkin diagram are given by $e^{\frac{k \pi i}{n_c}}+e^{\frac{-k \pi i}{n_c}}$, where $k$ is an exponent with the same multiplicity. 

Kac asked the question, whether the Coxeter exponents for the $ADE$ quivers can be generalized beyond $SU(2)$ theory in a talk given by Terry Gannon in 1994 at MIT \cite{Kac94q}. Gannon commented in his lecture at the Shanks workshop at Vanderbilt University in 2017, that if such a generalization exists, then the correspondence between the Coxeter exponents and spectrum of the adjacency matrices should work for all quivers of module categories acted on by the representation category of quantum $\ssl_n$, and it may even be true over any simple Lie algebra $\g$ at level $\ell$ \cite{Gan17}. These quivers were called higher Dynkin diagrams by Ocneanu in \cite{Ocn17}.

\color{black}

In this paper, we generalize the Coxeter exponents and prove such a correspondence for any (graded) unital *-representation $\Pi$ of the Verlinde algebra \cite{Ver88} of any simple Lie algebra $\g$ at any level $\ell$, including the generalized Dynkin diagrams in \cite{FraZub90} and higher Dynkin diagrams over quantum $\ssl_n$. This answers positively the questions of Kac and Gannon, when $\g$ is of type $ADE$. A modified correspondence also holds for other types. Our main results are summarized in \S \ref{Sec: Summary}.

We study the notions related to affine Lie algebras and quantum groups using Lie groups and their subgroups. We consider a quantum group as a simple Lie algebra $\g$ with a level $\ell$. We construct its Verlinde algebra by {\it anti-symmetric characters} defined in \S \ref{Sec: Anti-symmetric Character}. For each level $\ell$, the anti-symmetric characters are defined by the Weyl denominators on a domain $T_{\ell,0}$, a subset of the maximal torus of the corresponding Lie group. From the choice of the domain, we obtain a natural cutoff of the fusion rule of representations from the Lie algebra $\g$ to the quantum group at level $\ell$, also known as the Wess-Zumino-Witten cutoff~\cite{WZ71, Nov81, Witten83, Witten84}. 
We attempt to understand the connection between the McKay correspondence and the quantum McKay correspondence with this approach.
Another motivation for introducing the anti-symmetric characters is understanding the fusion rule and their generating functions in a closed form for the representations of two families of quantum subgroups constructed in \cite{Liu15}.

Additive functions on Auslander-Reiten quivers were studied by Gabriel in \cite{Gab80}. 
For an $ADE$ Dynkin diagram, the root system can be realized as additive functions on the Auslander-Reiten quiver, and the Coxeter element is given by a translation functor. The adjacency matrix of the $ADE$ Dynkin diagram can be extended to a $\Z_2$-graded unital *-representation of the Verlinde algebra of $\ssl_2$ at level $\ell$, and $\ell=n_c-2$.
We generalize related concepts for any (graded) unital *-representation $\Pi$ of the Verlinde algebra of a simple Lie algebra $\g$ at level $\ell$, and we introduce the corresponding adjacency matrices, quantum root spaces, quantum Coxeter elements and quantum Coxeter exponents. See the end of \S \ref{Sec: quantum Coxeter exponents} for a dictionary. 

We study the spectral theory using the Fourier duality in \S \ref{Sec: Fourier Duality}.
We apply the Fourier transform to diagonalize the actions of the adjacency matrices, and identify their simultaneous eigenvalues as elements in $T_{\ell,0}$ modulo the Weyl group action, which we call the {\it spectrum}.
On the other hand, we apply the Fourier transform to diagonalize the actions of quantum Coxeter elements, and identify their simultaneous eigenvalues as elements in $T_{\ell,0}$. Moreover, we introduce the quantum Coxeter exponents for the elements in $T_{\ell,0}$.
Then we can compare their multiplicities using these two identifications and obtain an equality in Theorem \ref{thm: exponents}. This equality generalizes the correspondence of Coxeter exponents for the $ADE$ Dynkin diagrams, and in this theorem we answer the questions of Kac and Gannon.

Gabriel constructed the root category using the quiver representations of $ADE$ Dynkin diagrams in \cite{Gab72}.
Using the root category, Ringle constructed the positive part of the corresponding Lie algebras in his seminal work \cite{Rin90}.
Happel gave another construction of the root category using 2-periodic derived categories of quiver representations in \cite{Hap87,Hap88}. Remarkably, Peng and Xiao constructed simple Lie algebras in \cite{PenXia97} and symmetrizable Kac-Moody Lie algebras using these derived categories in  \cite{PenXia00}.
\textcolor{black}{Dorey constructed the $ADE$ root system and the Coxeter element using the corresponding nimrep data of quantum $\ssl_2$ in \cite{Dor93}.}
Ocneanu proposed a construction of the root category using the module category over quantum $\ssl_2$ in \cite{Ocn02,Ocn17}, see further discussion by Kirillov and Thind in \cite{KirThi10,KirThi11}. 
Furthermore, Kirillov and Thind constructed the corresponding derived category in \cite{KirThi12}. It will be interesting to generalize the construction of the Lie algebras for quivers over a quantum group $\g$ beyond $\ssl_2$.

\color{black}
Around 1986, Kac observed a correspondence between the diagonal of an $\ssl_2$ modular invariance and the Coxeter exponents of $E_6$, which turned out to be a general phenomenon for all $\ssl_2$ modular invariances as discussed in \cite{CIZ87a, CIZ87, Kat87}.  
It is natural to ask for a generalization of this correspondence over a simple Lie algebra $\g$ at level $\ell$ as well. 
Kac and Gannon asked whether for any module category of the representation category of quantum $\ssl_n$, there is a correspondence between the diagonal of its modular invariance and the spectrum of its adjacency matrices. 
We hope to answer this question in the future.
\color{black}

\begin{ac}
The authors would like to thank Adrian Ocneanu for his motivating course Physics 267 at Harvard in the 2017 fall term \cite{Ocn17}.
The authors would like to thank Arthur Jaffe for much help to complete this paper. 
The authors would like to thank Robert Coquereaux, David Evans, Terry Gannon, Vaughan Jones, Victor Kac, Jie Xiao and Jean-Bernard Zuber for helpful comments. 
Zhengwei Liu is supported by grants TRT 080 and TRT 159 from the Templeton Religion Trust. 
Jinsong Wu is supported by NSFC 11771413 and grant TRT 159.
\end{ac}

\section{Background}\label{Sec:Backgroupnd}

It is well known that the simple Lie algebras are classified by Dynkin diagrams as their underlying symmetry. The construction from Dynkin diagrams to Lie algebras was given by Chevalley. The simply-laced Dynkin diagrams are the $ADE$ Dynkin diagrams and one can obtain the others from an orbifold construction of the $ADE$ ones.

In 1980, McKay found a one-to-one correspondence between subgroups of $SU(2)$ and the affine ADE Dynkin diagrams \cite{Mck80}. The affine Dynkin diagrams appeared as the quivers of the irreducible representations (irreps) of $SU(2)$ tensoring the standard representation. The Mckay correspondence relates subgroups of $SU(2)$ and $ADE$ Lie theory.

Around 1968, Kac and Moody studied infinite dimensional Lie algebras, known as Kac-Moody Lie algebras, see \cite{Kac68, Moo68, Kac90}.
The type $A_{\ell+1}$ Dynkin diagrams appeared as the quivers of the semisimple irreps of the affine Lie algebra $\ssl(2)$ at level $\ell$.
In 1983, Jones classified the indices of a subfactor $\N \subset \M$, an inclusion of von Neumann algebras with trivial center  \cite{Jon83}:
$$
\left\{4\cos^2{\frac{\pi}{2+\ell}}, \ell=1,2,\cdots\right\} \cup [4,\infty].
$$
For each $\ell$, he constructed a subfactor with index $4\cos^2{\frac{\pi}{2+\ell}}$, whose principal graphs, namely the quiver of bimodules, is the $A_{\ell+1}$ Dynkin diagram. 
The $A_{\ell+1}$ Dynkin diagram also appeared as the quiver of semisimple irreps of the Drinfeld-Jimbo quantum group $U_q sl_2$, $q=e^{\frac{\pi i}{2+\ell}}$ \cite{Jim85,Dri86}. The correspondence between the two representation theories is given in \cite{Jon87}. 
Wassermann found another conceptual connection between subfactor theory and representation theory of quantum groups in conformal field theory \cite{Was98}. The correspondence between representations of affine Lie algebras and representations of quantum groups is given by Kazhdan and  Lusztig in \cite{KazLus91}. 
The $A_{\ell+1}$ Dynkin diagram appeared as a cutoff of $A_\infty$. This is a general phenomenon also known as the Wess-Zumino-Witten cutoff~\cite{WZ71, Nov81, Witten83, Witten84}.

\color{black}
In 1987, Cappelli, Itzykson and Zuber classified the modular invariances of quantum $\ssl_2$ at level $\ell$ by $ADE$ Dynkin diagrams with Coxeter number $2+\ell$.  The diagonals of the modular invariants matched the multiplicities of the Coxeter exponents, which was first observed by Kac for the $E_6$ case, see further discussion in \cite{CIZ87}. A connection between the diagonals of the modular invariances and the spectrum of the adjacency matrices of representations of the Verlinde algebra of quantum $\ssl_n$ has been studied by Di Francesco and Zuber in \cite{FraZub90}.
In 1988, Ocneanu outlined the surprising $A_{n}, D_{2n}, E_6, E_8$ classification of the standard invariants of subfactors with index less than 4 in \cite{Ocn88}. The proof of the classification is given in \cite{GdlHJ89,BN91,Izu91,Izu94,Kaw95b}. This $ADE$ classification is also a classification of subfactors due to Popa's reconstruction theorem \cite{Pop94}. See the survey paper \cite{JMS14} for more details. In 1994, Ocneanu proposed a connection between subfactors and extended Turaev-Viro TQFT \cite{Ocn94}.
Given a subfactor $\N\subset \M$, if the $\N-\N$ bimodule category is isomorphic to the representation category of quantum $\ssl_2$, which defines a Turaev-Viro TQFT \cite{TurVir92}, then the subfactor defines an extended TQFT. Ocneanu considered this subfactor  (or the $\M-\M$ bimodule category) as a subgroup of quantum $\ssl_2$ and the $\N-\M$ module category as a module of quantum $\ssl_2$, see \cite{Ocn00}.
All $ADE$ Dynkin diagrams appeared as the quiver of the generating $\N - \N$ bimodule, corresponding to the fundamental representation of quantum $\ssl_2$, acting on irreducible $\N-\M$ bimodules.
The correspondence between $ADE$ subfactors and modular invariance is given by B\"ockenhauer, Evans and Kawahigashi based on the $\alpha$-induction \cite{LonReh95,Xu98,BEK99,BEK00,Ocn00}.
A corresponding categorical formalization of module categories and the classification has been done by Kirillov and Ostrik with independent proofs in \cite{KirOst02,Ost03}.
The $ADE$ classification related to quantum $\ssl_2$ has become known as the quantum McKay correspondence. 
\color{black}

Inspired by Chevalley's construction and McKay correspondence, one can consider the $ADE$ Lie theory as a mathematical theory over $\ssl_2$. It turns out to be natural to study the notions in Lie theory over quantum $\ssl_2$ and other quantum groups.
In this direction, Zuber introduced generalized Dynkin diagrams over $\ssl_n$ at level $\ell$ and a generalization of Coxeter elements and Coxeter exponents in \cite{Zuber97, Zuber98}. Many examples of generalized Dynkin diagrams appeared in conformal field theory, see the book of Di Francesco, Mathieu, and S\'en\'echal \cite{FMS}. 
Ocneanu reformulated the generalized Dynkin diagrams as the quivers of modules of $\ssl_n$ at level $\ell$ and proposed a classification for $\ssl_2$, $\ssl_3$ and $\ssl_4$ in \cite{Ocn02}, and he called these quivers higher Dynkin diagrams \cite{Ocn17}. 
Another approach is given by Etingof and Khovanov in terms of integer modules over the Verlinde algebras of quantum groups \cite{EtiKho95}.

Zuber's motivation for studying his generalized Dynkin diagrams arises from conformal field theory \cite{Zuber97,Zuber98}.
Xu constructed type $E$ module categories of the representation category of quantum groups as modules of commutative Frobenius algebras which arose from conformal inclusions in conformal field theory \cite{Xu98}. He also computed the corresponding quivers, namely higher Dynkin diagrams, for small rank quantum groups. 
The analogous construction for orbifolds, namely type $D$ module categories, was given in \cite{BE99} and the remaining $E_7$ modular invariant and module category was constructed in \cite{BEK00}. In general, it requires the modular invariance and 6j symbols to construct irreducible modules and to compute the quivers as shown in \cite{BEK99,BEK00,Ocn00}.
However, it remains a difficult problem to compute the 6j symbols in a closed form. It is also challenging to compute the corresponding quivers in closed forms when the quantum groups have large rank.  

Liu introduced a new type of Schur-Weyl duality for families of quantum subgroups in \cite{Liu15}. This provides new methods to compute the quivers of module categories without computing the modular invariance and 6j symbols. See further discussion in \cite{LiuRyb}. It would be interesting to compute the quantum Coxeter exponents and multiplicities for these families of examples.

For the $\ssl_3$ case, Di Francesco and Zuber investigated the McKay correspondence both in the classical and quantum sense in \cite{FraZub90}. 
Gannon classified the modular invariance for quantum $\ssl_3$ in \cite{Gan94}.
Ocneanu proposed a classification of unitary module categories of quantum $\ssl_3$ and the non-existence of the Ocneanu cell system of a nimrep in \cite{Ocn02}. The existence of the Ocneanu cell systems was computed by Evans and Pugh in \cite{EvaPug09}, and the non-existence of the Ocneanu cell system of a nimrep was proved by Coquereaux, Schieber and Isasi in \cite{CIS10}.
Recently Evans and Pugh classified modular invariances and unitary module categories for $SO(3)_{2m}$ in \cite{EvaPug18}.

Following the quantum McKay correspondence, Ocneanu proposed a generalization of the Lie theory in \cite{Ocn02} and gave a  course ``higher representation theory'' at Harvard in the 2017 fall term. For Ocneanu's blueprint shown in his course, we refer readers to the lecture notes in \cite{Ocn17}. Some motivating examples including higher Dynkin diagrams, higher roots, higher Coxeter elements over $\ssl_3$ were discussed in \cite{Ocn17}, see also hyper roots over $\ssl_3$ in \cite{Coq18}.

\section{Preliminaries}  

Let $\g$ be a simple complex {\bf Lie algebra}, that is the complexification of the Lie algebra $\mathfrak{k}$ of a simply-connected compact {\bf Lie group} $K$.
Let $\mathfrak{t}$ be a maximal abelian subalgebra of $\mathfrak{k}$, and let $T$ be the {\bf maximal torus}, a Lie subgroup of $K$ whose Lie algebra is $\mathfrak{t}$. 
Then $\mathfrak{h}=\mathfrak{t}+i\mathfrak{t}$ is a {\bf Cartan subalgebra} of $\g$.
Denote by $\langle \cdot, \cdot \rangle$ an inner product on $\g$ such that it is invariant under the adjoint action of $K$ and taking real values on $\mathfrak{k}$. Let $\rk$ be the {\bf rank} of $\g$.

Denote the set of {\bf roots} of $\g$ by $\bR$. 
Let $H_\rr$ be the \textbf{coroot} of $\rr$ for $\g$.
Denote by $\mathfrak{R}$ the $\Z$-linear space spanned by $\{H_\rr: \rr\in \bR\}$.
For any $\rr\in \bR$, we denote by $\theta_\rr$ the {\bf reflection} $\theta_\rr(\kk)=\kk-\langle \kk, H_\rr\rangle\rr$ for any $\kk\in\mathfrak{t}$.
Let $W$ be the {\bf Weyl group} of $\g$ generated by the reflections $\theta_\rr(\kk)$.
For any $\theta\in W$, we denote by $\varepsilon(\theta)$ the sign of $\theta$.
Let $\sr=\{\rr_1, \ldots, \rr_{\rk}\}$ be the set of \textbf{simple roots} in $\bR$.
Any root $\displaystyle \rr=\sum_{j=1}^{\rk}k_j\rr_j$ with $k_j\in \mathbb{N}$ is a {\bf positive root}. 
The set of positive roots is denoted by $\bR_+$.
Let $\rr_+$ be the {\bf highest positive root}, namely $\rr_+ +\rr$ is not a root for any $\rr \in \sr$.
We assume that the inner product $\langle \cdot, \cdot\rangle$ is normalized such that $\langle \rr_+, \rr_+\rangle =2$.
Let $\displaystyle \rr_+=\sum_{j=1}^{\rk} a_j\rr_j$ and $\displaystyle H_{\rr_+}=\sum_{j=1}^{\rk} a_j^\vee \rr_j$.
Then $\displaystyle c_\g=1+\sum_{j=1}^{\rk} a_j$ is the {\bf Coxeter number} of $\g$ and $\displaystyle c_\g^\vee=1+\sum_{j=1}^{\rk}a_j^\vee$ is the {\bf dual Coxeter number}.
We denote by $\sR$ the \textbf{root lattice}, 
$$\sR=\left\{\sum_{j=1}^{\rk} k_j\rr_j : k_j\in\Z \right\}.$$

Let $\fw=\{\ww_1, \ldots, \ww_{\rk}\}$ be the set of \textbf{fundamental weights}, such that $\langle \ww_j, H_{\rr_k}\rangle=\delta_{j,k}$ for any $1\leq j, k\leq \rk$.
Denote by $\sW$ the $\Z$-linear space spanned by the fundamental weights.
Then $\sW$ is the \textbf{weight lattice} of $\g$, and $\sW$ is an additive group.
The root lattice $\sR$ is a subgroup of $\sW$.
It is known that 
$$Z(\g):=\sW/\sR$$ is isomorphic to the {\bf center} $Z(K)$ of the Lie group $K$. Denote its order by $n_z:=|Z(\g)|$.
Let $\mathcal{C}$ be the \textbf{closed fundamental Weyl chamber} of $\sW$, i.e.
$$\mathcal{C}=\left\{\kk=\sum_{j=1}^{\rk}k_j\ww_j: k_j\geq 0, 1\leq j\leq \rk\right\}.$$
For any $\kk\in \mathcal{C}$, we denote by $\mathcal{W}(\kk)$ the {\bf weight diagram} of the irreducible representation with the highest weight $\kk$.
The \textbf{Weyl vector} $\rho$ is the sum of fundamental weights, 
$$\rho=\sum_{j=1}^{\rk}\ww_j.$$

For any $\kk\in\mathcal{C}$, we denote $V_{\kk}$ the {\bf irreducible representation} of $K$ with highest weight $\kk$.
Let $\chi_\kk$ be the character of the Lie group $K$ associated to $V_{\kk}$.
Suppose $\displaystyle V_{\jj}\otimes V_{\kk}=\bigoplus_{\ssb \in \mathcal{C}} N_{\jj,\kk}^{\ssb} V_{\ssb}$, where $N_{\jj,\kk}^{\ssb}\in \mathbb{N}$ is called a {\bf fusion coefficient}.
Then 
\begin{equation}\label{Equ:fusion}
\displaystyle \chi_{\jj}\chi_{\kk}=\sum_{\ssb \in \mathcal{C}} N_{\jj,\kk}^{\ssb} \chi_{\ssb}.
\end{equation}

For any $\kk\in \mathcal{C}$ and $H\in\mathfrak{t}$, the {\bf Weyl character formula} states
\begin{align*}
\chi_{\kk}(e^{H})
&=\frac{\sum_{\theta\in W}\varepsilon(\theta) e^{i\langle \theta(\kk+\rho), H\rangle}}{\sum_{\theta\in W}\varepsilon(\theta) e^{i\langle \theta(\rho), H\rangle }}.
\end{align*}

For a level $\ell\in \mathbb{N}$, let $n_c=c_\g^\vee+\ell$ be the {\bf quantum Coxeter number}.
Define
\begin{align*}
\rooth&:=\{ \theta(\rr_+) | \theta \in W\};\\
\sRa&:=\text{span}_{\Z}\{ \theta(\rr_+), \theta \in W \};\\
\sRl&:=n_c\sRa. 
\end{align*}

For any $\kk\in \sW$, define a {\bf translation} $\vartheta_\kk$ on $\sW$ by 
$$\vartheta_\kk(\jj)=\kk+\jj, \quad \jj\in \sW.$$
Then the set $\{\vartheta_\kk\}_{\kk\in \sW}$ is a free abelian group.
For a level $\ell \in \mathbb{N}$, 
the \textbf{affine Weyl group} $\widehat{W}$ is generated by the Weyl group $W$ and the translation $\vartheta_{n_c \rr_+}$.
The translation subgroup $W_t$ of $\widehat{W}$ is given by 
\begin{align*}
W_t:=\{\vartheta_{\rr} \; | \; \rr \in \sRl\}. 
\end{align*}
Then $\widehat{W}=W_t \rtimes W$.
The {\bf alcove} $\mathcal{C}_\ell$ is defined as
$$\mathcal{C}_\ell=\left\{\kk=\sum_{j=1}^{\rk} k_j\ww_j \in \mathcal{C}: k_j>0, j=1,\ldots, \rk, \langle \kk, H_{\rr_+} \rangle <n_c\right\}\subset\mathcal{C}.$$

The finite dimensional semisimple irreducible representations of the quantum group $\g$ at level $\ell$ are given by 
$\{V_{\kk} \; | \; \kk \in \mathcal{C}_{\ell}-\rho \}$.
Their fusion rule, also known as the Wess-Zumino-Witten cutoff, is
\begin{equation}\label{Equ: quantum fusion}
V_{\kk} \otimes V_{\jj}=\sum_{\ssb' \in \mathcal{C}, \hat{\theta} \in \widehat{W}, \atop \hat{\theta}(\ssb'+\rho)=\ssb+\rho} \varepsilon(\theta) N_{\kk, \jj}^{\ssb'} V_{\ssb},
\end{equation}
where $\varepsilon(\theta)=\pm 1$ is the sign of $\theta$.
The fusion algebra is known as the {\bf Verlinde algebra} of the quantum group \cite{Ver88}.

For $ADE$ Dynkin diagrams, Gabriel constructed the root system using additive functions on the Auslander-Reiten quiver  \cite{Gab80}, see also \cite{Hap87}. We briefly recall this construction and formalize related concepts in a general situation in \S \ref{Sec: quantum Coxeter exponents}. Take an $ADE$ Dynkin diagram $G$ as a bipartite graph and let $\varepsilon : G_v \to \Z_2$ be a $\Z_{2}$ grading of the vertices $G_v$ of $G$. The vertices of the {\bf Auslander-Reiten quiver} is  
$\Gamma=\{ (i, x) \in \Z_{2n_c} \times G_v: i+\varepsilon(x)=0 \in \Z_2  \}.$ 
A function $f$ on $\Gamma$ is called {\bf additive}, if $f(i,x) \in \Z$ and 
\begin{align}\label{Equ: additive}
f(i, x)+f(i+2,x)&=\sum_{y\in N(x)} f(i+1,y), \; \forall \; (i,x)\in \Gamma \;,
\end{align}
where $N(x)$ is the set of vertices adjacent to $x$. Let $\hil$ be the space of additive functions with the discrete measure on $\Gamma$, $P_{\hil}$ be the orthogonal projection from $L^2(\Gamma)$ onto $\hil$, and $\delta_{i,x}$ be the delta function at $(i,x)$. Then 
$2n_c P_{\hil}(\delta_{i,x})$ is an additive function. Moreover, 
$$\{\sqrt{2n_c} P_{\hil}(\delta_{i,x}) : (i, x) \in \Gamma \}$$
forms the root system of $G$. Furthermore the translation $\vartheta$ on $\Gamma$, $\vartheta(i,x)=(i+2,x)$, induces an Coxeter transformation on the root system.

\section{Fourier duality}\label{Sec: Fourier Duality}

In Lie theory, the exponential map $H\to e^{2 \pi H}$ is a group isomorphism 
$\mathfrak{t}/\mathfrak{R}\cong T$.
Therefore the Fourier transform $\hat{} : \kk\to \hat{\kk}$, given by
\begin{equation}\label{Equ:F}
\hat{\kk}(e^H)=e^{2 \pi i \langle \kk, H \rangle}\;,
\end{equation}
is well defined on $e^H \in T$. 
The map $~\hat{}~$ is a group isomorphism from $\sW$ to the dual of the abelian group $T$.
The weight lattice $\sW$ is the dual of the coroot lattice $\mathfrak{R}$:
\begin{itemize}
\item[(1)] $\kk\in \sW$ iff $\langle \kk, H \rangle \in \mathbb{Z}$, $\forall H \in \mathfrak{R}$;
\item[(2)] $H\in \mathfrak{R}$ iff $\langle \kk, H \rangle \in \mathbb{Z}$, $\forall \kk \in \sW$.
\end{itemize}
The following result is well-known in Lie theory:

The weight lattice $\sW$ is a free abelian group. Let $A$ be a subgroup of $\sW$, then $A$ is a free abelian group. 
\begin{definition}\label{Def:TA}
When $\sW/A$ is finite,  
we define the dual lattice of $A$ in $\mathfrak{t}$ as
$$\mathfrak{t}_{A}=\{H\in \mathfrak{t} | \langle  \rr, H \rangle\in \mathbb{Z}, \forall \rr\in A \}.$$
Define the corresponding subgroup of $T$ as
$$T_{A}=\left\{e^{2\pi H} \in T: H \in \mathfrak{t}_{A} \right\}.$$
\end{definition}
By the duality of lattices, we have the following result:
\begin{proposition}\label{prop: dual lattice}
For any $\rr \in \mathfrak{t}$,
$\rr \in A$ iff $\langle \rr, H \rangle \in \mathbb{Z}$, $\forall H \in \mathfrak{t}_{A}$.
\end{proposition}

\begin{theorem}\label{Thm: group duality}
The map $~\hat{}~$ induces a group isomorphism from $\sW/A$ to the dual of the abelian group $T_{A}$.
\end{theorem}
\begin{proof}
If $\kk \in A$, then by the definition of $T_{A}$, we have 
$\hat{\kk}(e^H)=1$,  for any $e^H\in T_{A}$.
So $~\hat{}~$ is well-defined from $\sW/A$ to the dual of $T_{A}$.
Conversely, for any $\kk \in \sW$,  if $\hat{\kk}(e^{2 \pi H})=1, ~\forall ~ H \in \mathfrak{t}_{A}$,
then $\langle \kk, H \rangle \in \Z$.
By Proposition \ref{prop: dual lattice}, $\kk\in A$.
So $~\hat{}~$ is injective.
Suppose that $f$ is a character of $T_A$. Then $f(e^H)=e^{i\langle \kk,H\rangle}$, for some $\kk\in \mathfrak{t}$.
For any $H\in 2\pi \mathfrak{R}$, we have $e^H=1 \in T_{A}$. So $f(e^H)=1$ and $\langle \kk,H\rangle\in2\pi\Z$.
By the duality between the lattices $\sW$ and $\mathfrak{R}$, we have $\kk\in \sW$.
So the map $~\hat{}~$ is surjective. 
Therefore $~\hat{}~$ is a group isomorphism from $\sW/A$ to the dual of the abelian group $T_{A}$.
\end{proof}

For any $\ell \in \mathbb{N}$, consider the case $A=\sRl=n_c \sRa$. 
\begin{definition}
The subgroup $\mathfrak{t}_{\ell}$ of the free abelian group $\mathfrak{t}$ is
$$\mathfrak{t}_{\ell}=\{H\in \mathfrak{t} | \langle  \rr, H \rangle\in \mathbb{Z}, \forall \rr\in \sRl\}.$$
The subgroup $T_\ell$ of $T$ is
$$T_\ell=\left\{e^{2\pi H}: H \in \mathfrak{t}_{\ell} \right\}\subset T.$$
The quotient group $\sWl:=\sW/ \sRl$ is a weight torus.
\end{definition}

\begin{corollary}
The map $~\hat{}~$ induces a group isomorphism from $\sW_{\ell}$ to the dual of the abelian group $T_{\ell}$.
\end{corollary}

\begin{corollary}
The order of the group $T_{\ell}$ is 
$$|T_{\ell}|=|\sWl|=|\sW/\sR| \cdot |\sR/ \sRa| \cdot |\sRa/\sRl|=n_z n_c^{\rk} |\sR/ \sRa|.$$
\end{corollary}

\begin{remark}
When $\g=\mathfrak{sl}_n$, $\sRl=n_c \sR$. So 
$|T_{\ell}|=|\sWl|=n (n+\ell)^{n-1}$.
\end{remark}

Let $L^2(\sWl)$ be the complex $L^2$ functions on $\sWl$ with counting measure.
Let $L^2(T_{\ell})$ be the complex $L^2$ functions on $T_{\ell}$ with Haar measure.
Recall that $\sWl$ and $T_{\ell}$ are dual to each other. 
For any $f\in L^2(\sWl)$, its Fourier transform $\mathcal{F}(f)$ in $L^2(T_{\ell})$ is 
$$\mathcal{F}(f)(e^H)=\sum_{\kk\in \sWl}f(\kk) e^{i\langle \kk, H\rangle}, \quad e^H\in T_{\ell}.$$
Then $\mathcal{F}$ is a unitary transformation:
$$\langle \mathcal{F}(f), \mathcal{F}(f')\rangle=\langle f,f'\rangle, \quad f, f'\in L^2(\sWl).$$
For any $g\in L^2(T_{\ell})$, the inverse Fourier transform $\mathcal{F}^{-1}(g)$ of $f$ is 
$$\mathcal{F}^{-1}(g)(\kk)=\frac{1}{|T_{\ell}|}\sum_{e^H\in T_\ell}g(e^H) e^{-i\langle \kk, H\rangle}, \quad \kk\in \sWl.$$

For any $f\in L^2(\sWl)$ and $\theta\in W$, we define $\theta(f)$ as $\theta(f)(\kk)=f(\theta^{-1}(\kk))$ for any $\kk\in \sWl$.
\begin{definition}
Let $L^2(\sWl)^W$ be the space of all anti-symmetric functions on $\sWl$:
$$L^2(\sWl)^W=\{f \in L^2(\sWl) : \theta(f)=\varepsilon(\theta) f, \ \forall \theta\in W\}.$$
\end{definition}
For any $g\in L^2(T_\ell)$ and $\theta\in W$, we define $\theta(g)$ as $\theta(g)(e^H)=g(e^{\theta^{-1}(H)})$ for any $e^H\in T_\ell$.
\begin{definition}
Let $L^2(T_\ell)^W$ be the space of all anti-symmetric functions on $T_\ell$:
$$L^2(T_\ell)^W=\{g\in L^2(T_\ell) : \theta(g)=\varepsilon(\theta) g, \ \forall \theta\in W\}.$$
\end{definition}

\begin{proposition}\label{prop: anti}
We have
$$\mathcal{F}(L^2(\sWl)^W)=L^2(T_\ell)^W.$$
\end{proposition}

\begin{proof}
Suppose $f$ is an anti-symmetric function in $L^2(\sWl)$.
Then for any $e^H\in T_{\ell}$ and $\theta\in W$,
\begin{align*}
\theta(\mathcal{F}(f))(e^{H})&=\sum_{\kk\in \sWl} f(\kk) e^{i\langle \kk, \theta^{-1}(H)\rangle}
=\sum_{\kk\in \sWl} f(\kk) e^{i\langle \theta(\kk), H\rangle}\\
&=\sum_{\kk\in \sWl} f(\theta^{-1}(\kk)) e^{i\langle \kk, H\rangle}
=\sum_{\kk\in \sWl} \varepsilon(\theta)f(\kk) e^{i\langle \kk, H\rangle}\\
&=\varepsilon(\theta)\mathcal{F}(f)(e^H).
\end{align*}
Hence $\mathcal{F}(f)$ is anti-symmetric.

Conversely, if $g$ is an anti-symmetric function in $L^2(T_{\ell})$, then 
$\mathcal{F}^{-1}(g)$ is anti-symmetric by a similar computation. 
Therefore $\mathcal{F}$ is a unitary transformation from $L^2(\sWl)^W$ to $L^2(T_\ell)^W$.
\end{proof}

\section{Anti-symmetric $\ell$-characters}\label{Sec: Anti-symmetric Character}

In this section, we introduce anti-symmetric characters for a simple Lie algebra $\g$ and a level $\ell$. We construct a $C^*$-algebra of those characters which represents the Verlinde algebra of the corresponding quantum group. Therefore we call it the {\it $\ell$-character Verlinde algebra of $\g$ at level $\ell$.}

\begin{definition}
We say $\kk \in \sWl$ is in a mirror, if it is in
$M(\sWl):=\{\kk \in \sWl :  \theta_{\rr} (\kk)=\kk, \text{ for some } \rr\in \bR\}.$
We say $e^H\in T_{\ell}$ is in a mirror if it is in
$M(T_{\ell}):=\{e^H \in T_{\ell} :  e^{\theta_{\rr}(H)}=e^H, \text{ for some } \rr\in \bR\}.$
Furthermore,
$$\sW_{\ell, 0}=\sWl \backslash M(\sWl);$$
$$T_{\ell, 0}=T_{\ell} \backslash M(T_{\ell}).$$
\end{definition}

Note that the Weyl group $W$ action fixes $M(T_{\ell})$ and $M(\sWl)$, so it also fixes $\sW_{\ell, 0}$ and  $T_{\ell, 0}$.
Moreover, the action of $W$ is transitive on each orbit in $\sW_{\ell, 0}$ and  $T_{\ell, 0}$.
Recall that the alcove $\mathcal{C}_\ell$ is defined as
$$\mathcal{C}_\ell=\left\{\kk=\sum_{j=1}^{\rk} k_j\ww_j \in \mathcal{C}: k_j>0, j=1,\ldots, \rk, \langle \kk, H_{\rr_+} \rangle <n_c\right\}\subset\mathcal{C}.$$
In affine Lie algebras, it is known that $\mathcal{C}_\ell$ is a fundamental domain of  $\sW_{\ell, 0}$  under the action of $W$.

\begin{remark}
If the Lie algebra $\g=\mathfrak{sl}_n$, then $|\CL|$ is the binomial coefficient
$$|\CL|=  \binom{n_c-1}{n-1}.$$
\end{remark}

\begin{definition}[$\ell$-characters]
For any $\kk\in \sWl$, $e^H\in T_\ell$, we define $\hchi_\kk$ at $\kk$ by 
$$\hchi_{\kk}=\mathcal{F}(\sum_{\theta\in W} \varepsilon(\theta) \delta_{\theta(\kk)} ),$$
where $\delta_{\kk}$ is 1 on $\kk$ and 0 elsewhere.
Then
$$\hchi_\kk(e^{H})=\sum_{\theta\in W} \varepsilon(\theta) e^{i \langle \theta(\kk), H\rangle}.$$
We define the $\ell$-character $\wchi_{\kk}$ to be the restriction of $\hchi_\kk$ on $T_{\ell,0}$.
\end{definition}

\begin{proposition}\label{prop:asym}
The function $\hchi$ is anti-symmetric on $\sWl$ and $T_{\ell}$, namely, for any $\kk\in \sWl$, $\theta\in W$,
\begin{align*}
\hchi_{\theta(\kk)}&=\varepsilon(\theta) \hchi_{\kk}, \\
\theta(\hchi_\kk)&=\varepsilon(\theta)\hchi_{\kk}.
\end{align*}
Consequently, $\hchi_{\kk}$ is supported in $T_{\ell,0}$.
\end{proposition}

\begin{proof}
By the definition of $\hchi$, it is anti-symmetric on $\sWl$.
By Proposition \ref{prop: anti}, $\hchi$ is anti-symmetric on $T_{\ell}$.
\end{proof}

\begin{corollary}\label{cor: mirror}
If $\kk$ is in a mirror, then $\hchi_\kk=0$.
For any $\kk\in \sWl$, if $e^H$ is in a mirror, then $\hchi_\kk(e^H)=0$.
\end{corollary}
\begin{proof}
By Proposition \ref{prop:asym}, if $\theta_\rr(\kk)=\kk$ for some $\rr\in \bR$, then
$$\hchi_{\kk}=\hchi_{\theta_\rr(\kk)}=\varepsilon(\theta_{\rr})\hchi_{\kk}=-\hchi_{\kk},$$
and we obtain $\hchi_{\kk}=0$.
For any $\kk\in \sWl$ and $e^{\theta_{\rr}(H)}=e^H$ for some $\rr\in \bR$, we have
$$\hchi_{\kk}(e^H)=\hchi_{\kk}(e^{\theta_{\rr}(H)})=-\hchi_{\kk}(e^H),$$
and hence $\hchi_{\kk}(e^H)=0$.
\end{proof}

\begin{theorem}
The set $\{|W|^{-\frac{1}{2}}\hchi_{\kk}\}_{\kk\in\mathcal{C}_\ell}$ forms an orthonormal basis (ONB) of $L^2(T_\ell)^W$.
In particular, $$\text{dim }L^2(T_\ell)^W=|\mathcal{C}_\ell|.$$
\end{theorem}

\begin{proof}
Note that 
$\{ |W|^{-1/2} \sum_{\theta\in W} \varepsilon(\theta) \delta_{\theta(\kk)} \}_{\kk \in \CL}$
form an orthonormal basis of $L^2(\sWl)^{W}$.
By the definition of $\hchi$ and Proposition \ref{prop: anti}, 
$\{|W|^{-\frac{1}{2}}\hchi_{\kk}\}_{\kk\in\mathcal{C}_\ell}$ forms an orthonormal basis of $L^2(T_\ell)^W$.
\end{proof}

By anti-symmetry, any function in $L^2(T_\ell)^W$ is supported in $T_{\ell,0}$. So $L^2(T_\ell)^W \cong L^2(T_{\ell,0})^W$.
\begin{definition}
Define the space $L^2(T_{\ell,0})^W$ as anti-symmetric functions on $T_{\ell,0}$:
$$L^2(T_{\ell,0})^W=\{g\in L^2(T_{\ell,0}) : \theta(g)=\varepsilon(\theta) g, \ \forall \theta\in W\}.$$
\end{definition}

\begin{corollary}\label{cor: basis}
The set of multiples of $\ell$-characters $\{|W|^{-\frac{1}{2}}\wchi_{\kk}\}_{\kk\in\mathcal{C}_\ell}$ forms an orthonormal basis of $L^2(T_{\ell,0})^W$.
\end{corollary}

Note that $\hchi_\rho$ is the restriction of the Weyl denominator on $T_{\ell}$.
The following result is well-known in Lie theory: 
$$\hchi_\rho(e^H)=
e^{i\langle \rho, H \rangle} \prod_{\rr\in \bR_+} \left(1-e^{-i\langle \rr, H\rangle}\right)=
\prod_{\rr\in \bR_+} \left(e^{i\langle \rr, H\rangle/2}-e^{-i\langle \rr, H\rangle/2}\right),$$
for any $e^H\in T_\ell$, see \cite{Kac90}. Consequently, one has the following result:

\begin{corollary}
For any $e^H \in T_{\ell}$, $\hchi_\rho(e^H)\neq0$ iff $e^H\in T_{\ell,0}$. Equivalently, $\wchi_{\rho}$ is invertible in $L^2(T_{\ell,0})$.
\end{corollary}

Recall that $\wchi$ is the restriction of $\hchi$ on $T_{\ell,0}$. 
\begin{definition}
We define the multiplication $\star$ of $\wchi_\kk$ and $\wchi_\jj$ for any $\kk, \jj\in \sWl$ to be 
$$\wchi_\kk \star\wchi_\jj=\displaystyle \frac{\wchi_\kk\wchi_\jj}{\wchi_\rho}.$$
\end{definition}
\noindent Then $\wchi_\rho$ is the identity under this multiplication.
Recall that the fusion coefficients of the representations of $\g$ is given in Equation \eqref{Equ:fusion}.
\begin{definition}\label{def: Nkjs}
For $\kk,\jj,\ssb \in \CL$, we define the fusion coefficient $\tilde{N}_{\kk, \jj}^\ssb$ as
\begin{equation}
\tilde{N}_{\kk, \jj}^\ssb=\sum_{\ssb' -\rho \in \mathcal{C}, \theta \in W \atop \theta(\ssb')-\ssb\in \sRl} \varepsilon(\theta) N_{\kk-\rho, \jj-\rho}^{\ssb'-\rho}.
\end{equation}
\end{definition}

\begin{remark}
By Equation \eqref{Equ: quantum fusion}, the fusion rule of representations of quantum $\g$ at level $\ell$ is given by 
$$V_{\kk-\rho} \otimes V_{\jj-\rho} =\bigoplus_{\ssb \in \CL } \tilde{N}_{\kk, \jj}^\ssb V_{\ssb-\rho} .$$
Consequently, $\tilde{N}_{\kk,\jj}^{\ssb}\in \mathbb{N}$.
\end{remark}

\begin{theorem}\label{thm:fusionrule}
For any $\kk,\jj,\ssb \in \CL$, 
$$\wchi_{\kk} \star \wchi_{\jj}=\sum_{\ssb \in \mathcal{C}_{\ell}}  \tilde{N}_{\kk,\jj}^{\ssb} \wchi_{\ssb}.$$
Equivalently, 
$$\tilde{N}_{\kk,\jj}^{\ssb}=\frac{1}{|W|}\langle \wchi_{\kk} \star \wchi_{\jj}, \wchi_{\ssb}\rangle.$$
\end{theorem}

\begin{proof}
Note that $\displaystyle \frac{\wchi_{\kk}}{\wchi_{\rho}}=\chi_{\kk-\rho}$ on $T_{\ell, 0}$.
Then in $L^2(T_{\ell,0})$, one has
\begin{align*}
\wchi_\kk\star\wchi_\jj&=\frac{\wchi_\kk \wchi_\jj }{\wchi_\rho}
= \chi_{\kk-\rho} \chi_{\jj-\rho} \wchi_\rho \\
&= \sum_{\ssb'-\rho\in\mathcal{C}}N_{\kk-\rho, \jj-\rho}^{\ssb'-\rho}\chi_{\ssb'-\rho} \wchi_\rho \\
&= \sum_{\ssb'-\rho\in\mathcal{C}}N_{\kk-\rho, \jj-\rho}^{\ssb'-\rho}\wchi_{\ssb'} \\
&= \sum_{\ssb\in\CL}\tilde{N}_{\kk, \jj}^\ssb\wchi_{\ssb} .
\end{align*}
The last equality uses the anti-symmetry established in Proposition \ref{prop:asym}, $\wchi_{\ssb'}=\varepsilon(\theta) \wchi_{\ssb}$, when 
$\theta \in W$ and $\theta(\ssb')-\ssb\in \sRl$.
\end{proof}

Recall that $\rr_1,\cdots, \rr_{\rk}$ are simple roots in $\bR$.
Then $\{-\rr_1,\cdots, -\rr_{\rk}\}$ is also a set of simple roots. 
So there is an element $\Omega \in W$, such that $\Omega(\rho)=-\rho$.  
Moreover, $\Omega(\mathcal{C}_\ell)=-\mathcal{C}_\ell$.

\begin{remark}
If the Lie algebra $\g$ is $\ssl_n$, we have for any $\kk\in \CL$,
$$\varepsilon(\Omega)=(-1)^{\frac{n(n-1)}{2}}.$$
\end{remark}

\begin{definition}
Define an involution $*: \kk\mapsto \Omega(-\kk)$ on $\mathcal{C}_\ell$.
Then it is well defined on $\sWl$, and $\rho^*=\rho$.
\end{definition}

For any $\kk\in \sWl$, by the definition of $\wchi$ and its anti-symmetry, we have
\begin{align}\label{Equ: omega}
\wchi_{\kk^*}=\wchi_{\Omega(-\kk)}=\varepsilon(\Omega)\wchi_{-\kk}=\varepsilon(\Omega) \overline{\wchi_{\kk}}.
\end{align}
The involution * on $\CL$ induces an involution on $L^2(T_{\ell,0})^W$:
$$\wchi_{\kk}^*:=\wchi_{\kk^*}, ~\kk \in \CL.$$

\begin{proposition}
For any $\kk,\jj \in \CL$, 
$\tilde{N}_{\kk,\jj}^{\rho}=1$ iff $\kk=\jj^*$.
Moreover, * induces an anti-isomorphism on $L^2(T_{\ell,0})^W$ with multiplication $\star$. 
\end{proposition}

\begin{proof}
By Corollary \ref{cor: basis} and
 Theorem \ref{thm:fusionrule}, for any $\kk,\jj \in \CL$,
\begin{align*}
\tilde{N}_{\kk,\jj}^{\rho}
=|W|^{-1}\langle \wchi_{\kk} \star \wchi_{\jj}, \wchi_\rho \rangle
=|W|^{-1}\langle \wchi_{\kk} , \wchi_\rho \star \wchi_{\jj}^* \rangle
=|W|^{-1}\langle \wchi_{\kk} , \wchi_{\jj}^* \rangle
=\delta_{\kk,\jj^*}.
\end{align*}

Moreover,
\begin{align*}
(\wchi_{\kk} \star \wchi_{\jj})^*
=\varepsilon(\Omega) \overline{\wchi_{\kk} \star \wchi_{\jj}}
=\varepsilon(\Omega) \frac{ \overline{\wchi_{\kk}} \overline{ \wchi_{\jj}}}{\overline{\wchi_{\rho}}}
=\frac{ \wchi_{\kk}^* \wchi_{\jj}^*} {\wchi_{\rho}}
=\wchi_{\jj}^* \star \wchi_{\kk}^* .
\end{align*}
\end{proof}

\begin{corollary}
For any $\kk,\jj, \ssb \in \CL$, 
$\tilde{N}_{\kk,\jj}^{\ssb}=\tilde{N}_{\kk^*,\jj^*}^{\ssb^*}$.
\end{corollary}

Recall that the Weyl character $\chi_{\kk}$ is graded by $\kk$ in $\sW/\sR\cong Z(\g)$.
Therefore the coefficient $N_{\jj,\kk}^{\ssb}$ is non-zero only if the grading matches, namely $\jj+\kk-\ssb\in \sR$.

\begin{definition}
We define the grading of $\wchi_\kk$ to be $\kk-\rho$ in $\sW/\sR\cong Z(\g)$.
\end{definition}

By Definition \ref{def: Nkjs}, $\tilde{N}_{\jj,\kk}^{\ssb}$ is non-zero only if the grading matches. 
So the grading is additive under the multiplication $\star$.

A \textbf{fusion ring} $\mathfrak{A}$ is a ring over $\Z$ with a basis $\{x_0=1, x_1, \ldots, x_m\}$ such that 
\begin{itemize}
\item[1.] $x_jx_k=\sum_{s=0}^m N_{j,k}^s x_s$, $N_{j,k}^s \in \mathbb{N}$;
\item[2.] There exists an involution $*$ on $\{0,1,2,\ldots, m\}$, such that $N_{ij}^0=\delta_{i,j^*}$, and it induces an
anti-isomorphism of $\mathfrak{A}$, $x_{k}^*:=x_{k^*}$ and
$$x_{k}^*x_{j}^*=(x_{j}x_{k})^*=\sum_{s=0}^m N_{j,k}^s x_{s}^*,$$
namely, $N_{k^*,j^*}^{s^*}=N _{j,k}^{s}$.
\end{itemize}

\begin{definition}
Let $R_{\ell}$ denote the $Z(\g)$-graded fusion ring with basis $\{\wchi_{\kk}\}_{\kk\in \CL}$.
\end{definition}

Note that $\{\chi_{\ww} | \ww \in \fw \}$ is a multiplicative basis of the fusion ring of Weyl characters, so $\{\wchi_{\ww+\rho} | \ww \in \fw\}$ is a multiplicative basis of $R_{\ell}$.

\begin{proposition}
For any $g,g' \in L^2(T_{\ell,0})^W$, $\kk \in \CL$, we have
$$\langle \wchi_{\kk} \star g , g' \rangle=\langle g, \wchi_{\kk}^* \star g' \rangle.$$
\end{proposition}

\begin{proof}
By Equation \eqref{Equ: omega}, we have
\begin{align*}
\langle \wchi_{\kk} \star g , g' \rangle 
=&\langle \frac{\wchi_{\kk}}{\wchi_{\rho}} g , g' \rangle \\
=&\langle g, \overline{\left(\frac{\wchi_{\kk}}{\wchi_{\rho}}\right)} g' \rangle \\
=&\langle g, \frac{\wchi_{\kk^*}}{\wchi_{\rho}} g' \rangle \\
=&\langle g, \wchi_{\kk^*}\star g' \rangle.
\end{align*}
\end{proof}

As a consequence, $L^2(T_{\ell,0})^W$ is a faithful $Z(\g)$-graded, unital *-representation of the *-algebra $L^2(T_{\ell,0})^W$.
So  $L^2(T_{\ell,0})^W$ is an abelian $Z(\g)$-graded, unital $C^*$-algebra with the multiplication $\star$, and involution $*$.
\begin{definition}
We call the $Z(\g)$-graded, unital $C^*$-algebra $L^2(T_{\ell,0})^W$ the $\ell$-character Verlinde algebra.
\end{definition}

\begin{remark}
By Theorem \ref{thm:fusionrule}, one sees that $L^2(T_{\ell,0})^W$ is isomorphic to the Verlinde algebra of 
quantum $\g$ at level $\ell$, and $R_{\ell}$ is the fusion ring. Therefore one can consider the anti-symmetric $\ell$-characters as the characters of the corresponding irreps. See \cite{AndPar95} for a construction of the corresponding fusion category.
\end{remark}

\section{GUS-Representations}

Recall that the fusion ring $R_{\ell}$ of $\g$ at level $\ell$ has a $Z(\g)=\sW/\sR$ grading, and $\wchi_{\jj} \in R_{\ell}$ is graded by $\jj-\rho \in Z(\g)$. 
Moreover, $\{\wchi_{\jj}\}_{\jj \in \CL}$ is a basis of $R_{\ell}$ and $\{\wchi_{\ww+\rho}\}_{\ww \in \fw}$ is a multiplicative basis of $R_{\ell}$.

\begin{definition}[GUS-rep]\label{Def:gus}
For a finite dimensional Hilbert space $\hil$, 
a representation $\Pi: R_{\ell} \to \hom(\hil)$ is called a unital, *-representation, abbreviated as US-rep, if for any $\wchi_{\jj} \in R_{\ell}$, $\kk \in Z(\g)$, the following properties are satisfied:
\begin{itemize}
\item[(1)] $\Pi(\wchi_{\rho})=I$; 
\item[(2)] $\langle \Pi(\wchi_{\jj}) v, v' \rangle = \langle v,  \Pi(\wchi_{\jj}^*) v' \rangle$ for any $v, v'\in \hil$.
\end{itemize}

Furthermore, if $\displaystyle \hil=\bigoplus_{\kk \in Z(\g) } \hil_{\kk}$ is $Z(\g)$-graded, and 
\begin{itemize}
\item[(3)] $\Pi(\wchi_{\jj}) \hil_{\kk} \subseteq \hil_{\kk+\jj-\rho}, \; \forall \; \kk \in Z(\g),$
\end{itemize}
then $\Pi$ is called a $Z(\g)$-graded, unital, *-representation, abbreviated as GUS-rep.
Equivalently, $\Pi$ is a ($Z(\g)$-graded,) unital, *-representation of the $\ell$-character Verlinde algebra. 
\end{definition}

\begin{definition}
For a US-rep $\Pi$, 
an ONB $B$ of $\hil$ is called a $\mathbb{K}$ basis, $\mathbb{K}=\mathbb{C},~\mathbb{R},~\mathbb{Z}$, or $\mathbb{N}$, if $\Pi_{\ww}:=\Pi(\wchi_{\ww+\rho}) \in M_d(\mathbb{K})$, for any $\ww \in \fw$.
We call $B$ a full $\mathbb{N}$ basis, if $\Pi(\wchi_{\jj}) \in M_d(\mathbb{N})$, for any $\jj \in \CL$.  
If $\Pi$ is a GUS-rep, $\displaystyle B=\bigsqcup_{\kk  \in Z(\g)} B_\kk$ and each $B_\kk$ is an ONB of $\hil_{\kk}$.
then $B$ is called $Z(\g)$-graded.
\end{definition}

A quiver $\G$ consists of a set $\G_v$ of vertices, a set $\G_e$ of oriented edges, a function $s: \G_e\mapsto \G_v$ giving the start of the edge and another function $t: \G_e\mapsto \G_v$ giving the target of the edge.

\begin{definition}
We call a quiver $\G$ $\g$-graded, if there is a grading map $\varepsilon$, $\varepsilon: \G_v\to Z(\g)$ and $\varepsilon: \G_e\mapsto \fw$, such that for any $e\in \G_e$, it is true that $\varepsilon(s(e))+\varepsilon(e)=\varepsilon(t(e))$ in $Z(\g)$.
Moreover, the adjacency matrix $\Delta_\ww$, $\ww \in \fw$ is defined as a matrix acting on $\G_v$, whose $\alpha,\beta$ entry is the number of edges graded by $\ww$ from $\beta$ to $\alpha$.
\end{definition}

\begin{definition}[Quantum Dynkin diagrams]\label{def:dynkin}
Let $\G$ be a $\g$-graded quiver and $\ell\in \mathbb{N}$. We call $\G$ a quantum Dynkin diagram over $\g$ at level $\ell$,
if there is a $Z(\g)$-graded unital *-homomorphism $\Pi_G: R_\ell \to M_{\G_v}(\mathbb{Z})$, such that 
$$\Delta_{\ww}=\Pi_{\ww}, ~\forall~ \ww\in \fw.$$
Furthermore, we say $\G$ is natural if $\Pi(\wchi_{\kk}) \in M_{\G_v}(\mathbb{N})$, $\forall~ \kk\in \CL$.
We call $\G$ simple, if $\G$ is a connected quiver. 
\end{definition}

\color{black}
\begin{remark}
Similar notions without grading have been studied by Etingof and Khovanov in \cite{EtiKho95}.
\end{remark}
\color{black}

\begin{proposition}
If $\G$ is a quantum Dynkin diagram, then the *-homomorphism $\Pi_G$ is unique.
\end{proposition}

\begin{proof}
For any  $\ww \in \fw$, the adjacency matrix $\Pi_{\ww}=\Delta_{\ww}$ is determined by $\G$ by definition. Since $\{\wchi_{\ww+\rho} | \ww \in \fw\}$ is a multiplicative basis of $R_{\ell}$, $\Pi$ is uniquely determined by $\G$.
\end{proof}

The multiplication $\star$ of $\wchi_\kk$, $\kk \in \CL$, on $L^2(T_{\ell,0})^W$ defines a GUS-rep $\Pi_A$ of $R_{\ell}$.
Recall that $\{\frac{1}{|W|}\wchi_{\jj}\}_{\jj\in \CL}$ is an orthonormal basis of $L^2(T_{\ell,0})^W$. Acting on this basis, we have the regular reprensentation
$$\Pi_A(\wchi_\kk) \wchi_{\jj}:=\wchi_{\kk} \star \wchi_{\jj}=\sum_{\ssb\in \CL} \tilde{N}_{\kk,\jj}^{\ssb} \wchi_{\ssb}.$$
In particular $\Pi_A(\wchi_\kk)_{\ssb,\jj}=\tilde{N}_{\kk,\jj}^{\ssb}$.

\textcolor{black}{
The fusion graph of the Verlinde algebra with respect to the fundamental representations is a quantum Dynkin diagram, which is usually referred as the type A graph in different fomulations: }

\begin{definition}[Type $A$ quantum Dynkin diagrams]
For any simple Lie algebra $\g$ and level $\ell$, we define the type $A$ quantum Dynkin diagram $\mathcal{A}_{\ell}(\g)$ as follows: 
\begin{itemize}
\item[(1)] The vertices of $\mathcal{A}_{\ell}(\g)$ is $\{\wchi_{\kk} : \kk \in\CL\}$ graded by $Z(\g)$.
\item[(2)] For any $\ww \in \fw$, the multiplicity of the edge from $\jj$ to $\ssb$ graded by $\ww$ is $\tilde{N}_{\ww+\rho,\jj}^{\ssb}$.
\end{itemize}
\end{definition}

Let $\mathscr{C}$ be the representation category of quantum $\g$ at level $\ell$, which is a unitary modular tensor category. Let $\mathscr{M}$ be a module category of $\mathscr{C}$. We obtain a quiver $G_{\mathscr{M}}$ from the action of $\mathscr{C}$ on $\mathscr{M}$:
\begin{itemize}
\item[(1)] The vertices of $G_{\mathscr{M}}$ are (representatives of) irreducible modules in $\mathscr{M}$, denoted by $Irr_{\mathscr{M}}$.
\item[(2)] For any $\ww \in \fw$, and irreducible modules $m_1,m_2$ in $\mathscr{M}$, the multiplicity of the edge from $m_1$ to $m_2$ graded by $\ww$ is $\dim \hom_{\mathscr{M}} (V_{\ww} \otimes m_1, m_2)$.
\end{itemize}

\begin{proposition}
For any $\kk \in\CL$, and $m_1 \in Irr_{\mathscr{M}}$,
we define 
$$\Pi_{\mathscr{M}}(\wchi_{\kk}) m_1=\sum_{m_2 \in Irr_{\mathscr{M}}} \dim \hom_{\mathscr{M}} (V_{\kk-\rho} \otimes m_1, m_2) m_2.$$
Then $\Pi_{\mathscr{M}}$ is a US-rep and $Irr_{\mathscr{M}}$ is a natural basis.
\end{proposition}

\begin{proof}
Since $V_0$ is the trivial representation, $\Pi_{\mathscr{M}}(\wchi_{\rho})$ is the identity. 
The associativity of the action of $\mathscr{C}$ on $\mathscr{M}$ implies that $\Pi$ is a representation.
By the Frobenius reciprocity 
$$ \dim \hom_{\mathscr{M}} (V_{\kk-\rho} \otimes m_1, m_2) = \dim \hom_{\mathscr{M}} ( m_1, V_{\kk^*-\rho} \otimes m_2),$$
we have that $\Pi$ is a US-rep.
Furthermore,  $\dim \hom_{\mathscr{M}} (V_{\kk-\rho} \otimes m_1, m_2) \in \mathbb{N}$, so $Irr_{\mathscr{M}}$ is a natural basis.

\end{proof}

We prove the results in the rest of the paper for any US-rep or GUS-rep of $R_{\ell}$ for any simple Lie algebra $\g$ at any level $\ell \in \mathbb{N}$. In particular, they are true for quantum Dynkin diagrams over $\g$ at level $\ell$.
This provides a general theory in the study of quantum Dynkin diagrams.

\section{Spectrum of representations}\label{Sec: Spectrum of representations}

In this section, we investigate the spectral theory of a GUS-rep $\Pi$ of $R_{\ell}$ defined in Definition \ref{Def:gus}.
We give an explicit construction of the spectrum of the regular representation $\Pi_A$ in $T$.
We prove that the spectrum for a GUS-rep $\Pi$ is contained in the spectrum of the type $A$ quantum Dynkin diagram. 

\begin{definition}
Recall that the Weyl group $W$ acts on each orbit of $T_{\ell,0}$ transitively. 
Define $\spec: \cong T_{\ell,0}/W$.
\end{definition}

For any $e^H\in T_{\ell, 0}$, let $\delta_{e^H}$ be the delta function at $e^H$. 
For any $\kk\in \CL$, define
\begin{align}
\Lambda_{e^H}&=\sum_{\theta \in W} \varepsilon(\theta) \delta_{e^{\theta(H)}}.
\end{align}
Then $\Lambda_{e^H} \in L^2(T_{0,\ell})^W$. Moreover,
\begin{align*}
\Pi_A(\wchi_{\kk}) \Lambda_{e^H}
=\wchi_{\kk} \star \Lambda_{e^H}
=\frac{\wchi_{\kk} \Lambda_{e^H}}{\wchi_{\rho}} 
=\chi_{\kk-\rho}(e^H) \Lambda_{e^H} .
\end{align*}
So $\Lambda_{e^H}$ is a common eigenvector of $\Pi_A(\wchi_{\kk})$ with eigenvalue $\chi_{\kk-\rho}(e^H)$.

Note that $\Lambda_{e^\theta(H)}=\varepsilon(\theta)\Lambda(e^H)$, so $\mathbb{C}(\Lambda_{e^{H}})$ is a well-defined eigenspace for $e^H\in \spec$.
The Weyl character $\chi_{\kk-\rho}$ is symmetric, so it is also well-defined on $e^H \in T_{\ell,0}/W$.

\begin{lemma}\label{Lem: spec}
For any $e^H, e^{H'} \in T_{\ell,0}$, the following are equivalent:
\begin{itemize}
\item[(1)] $e^H=e^{\theta(H')}$, for some $\theta\in W$;
\item[(2)] $\chi_{\kk}(e^H)=\chi_{\kk}(e^{H'}),$ for any $\kk \in \mathcal{C}$;
\item[(3)] $\chi_{\ww}(e^H)=\chi_{\ww}(e^{H'}),$ for any $\ww \in \fw$.
\end{itemize}

\end{lemma}

\begin{proof}
(1) $\rightarrow$ (2): It follows from the fact that Weyl characters are symmetric under the Weyl group action.

(2) $\rightarrow$ (1):
The Weyl characters $\{\chi_{\kk} \}_{\kk \in \mathcal{C}}$ are a basis of symmetric functions on the maximal torus $T$ with respect to the Weyl group $W$ action. 
So any symmetric function has the same value on $e^H$ and $e^{H'}$.
By the Stone-Weierstrass theorem, $e^H=e^{\theta(H')}$, for some $\theta \in W$.

(2) $\leftrightarrow$ (3): It follows from the fact that $\{\chi_{\ww} | \ww \in \fw\}$ is a multiplicative basis of Weyl characters.

\end{proof}

\textcolor{black}{The simultaneous eigenvalue of the adjacency matrices of the Verlinde algebra was captured by the modular $S$ matrix, known as the Verlinde formula \cite{Ver88}. Now we represent these simultaneous eigenvalues using $e^H \in \spec$:}

\begin{theorem}\label{thm: A spectrum}
When $\Pi=\Pi_A$,
for any $e^H\in \spec$, $\mathbb{C}(\Lambda_{e^{H}})$ is the common eigenspace of the adjacency matrices $\Pi_{\ww}$, $\ww\in \fw$, with eigenvalue $\chi_{\ww}(e^H)$. Any simultaneous eigenvalue of the adjacency matrices is of this form and it has multiplicity one.
\end{theorem}

\begin{proof}
We have shown that for any $e^H\in \spec$, $\mathbb{C}(\Lambda_{e^{H}})$ is the common eigenspace of the adjacency matrices $\Pi_{\ww}$, $\ww\in \fw$, with eigenvalue $\chi_{\ww}(e^H)$.
By Lemma \ref{Lem: spec}, if $e^H$ and $e^{H'}$ are different in $\spec$, then the corresponding eigenvalues are different.

Note that $\dim(L^2(\sWl)^W)=\frac{|W_{\ell,0}|}{|W|}$ and $\dim(L^2(T_{\ell})^W)=\frac{|T_{\ell,0}|}{|W|}$.
By Proposition \ref{prop: anti}, their dimensions are the same. So
$$|\spec|=\frac{|T_{\ell,0}|}{|W|}=\frac{|W_{\ell,0}|}{|W|}=|\CL|.$$
Therefore each eigenvalue corresponding to $e^H\in \spec$ has multiplicity one, and they are all eigenvalues.
\end{proof}

\begin{theorem}\label{thm:subspectrum}
Suppose $\Pi$ is a US-rep of $R_{\ell}$ over $\g$ at level $\ell$.
Then for any common eigenvector $v$ of the adjacency matrices $\Pi_{\ww}$, $\ww\in \fw$, there is $e^H\in \spec$ such that
$$\Pi_{\ww}v=\chi_{\ww}(e^H) v.$$
\end{theorem}

\begin{proof}
Since $L^2(T_{\ell,0})^W$ is an abelian $C^*$ algebra, all its one dimensional representations are sub-representations of the regular representation $\Pi_A$.
So any simultaneous eigenvalue of $\Pi$ is always a simultaneous eigenvalue of $\Pi_A$. By Theorem \ref{thm: A spectrum}, the statement holds.
\end{proof}

\textcolor{black}{The following result is well-known in different formulations, see e.g. \cite{EtiKho95}:}

\begin{theorem}
Quantum Dynkin diagrams over $\mathfrak{sl}_2$ at level $\ell\in \mathbb{N}$ are $ADE$ Dynkin diagrams with Coxeter number $2+\ell$.
\end{theorem}

\begin{proof}
If $\G$ is an $ADE$ Dynkin diagram, 
then they are $\Z_2$-graded quivers of modules of quantum $\mathfrak{sl}_2$.
So they are quantum Dynkin diagrams over $\mathfrak{sl}_2$ at level $\ell$, and $\ell+2$ is the Coxeter number of $\G$.

If $\G$ be a quantum Dynkin diagram over $\mathfrak{sl}_2$ at level $\ell$, then 
$\Pi_1$ is the adjacency matrix of the bipartite graph $\G$. 
By Theorem \ref{thm:subspectrum}, any eigenvalue of $\Pi_1$ is $e^{\frac{\pi i}{\ell+2}t}-e^{-\frac{\pi i}{\ell+2}t}$, for some $t\in \Z_{2(\ell+2)}$, $t\neq 0$ and $t\neq \ell+2$.
So $\|\Pi_1\|<2$. 
Therefore $\G$ is an $ADE$ Dynkin diagram, a well known result. Moreover, $R_{\ell}$ is determined by $\Pi_1$, so $\ell$ is determined.
\end{proof}

\begin{definition}
Suppose $\Pi$ is a US-rep of $R_{\ell}$ over $\g$ at level $\ell$.
For a common eigenvector $v$ of the adjacency matrices, we call the corresponding $e^H \in \spec$ in Theorem \ref{thm:subspectrum} the spectrum of $v$, denoted by $\text{sp}(v)$. 
\end{definition}

\begin{definition}
Suppose $\Pi$ is a US-rep of $R_{\ell}$ over $\g$ at level $\ell$.
Let $B$ be a common eigenbasis of the adjacency matrices.
We define the spectrum of $\Pi$ to be $\{$sp$(v), v \in B\}$, a subset of $\spec$. 
We define $B(e^H)$ to be the subset of $B$ with spectrum $e^H$.
The multiplicity of the spectrum $e^H\in \spec$ is the order $|B(e^H)|$, denoted by $m_{\Pi}(e^H)$.
\end{definition}
Since the Weyl Character is symmetric under the action of $W$, we can lift the spectrum from $\spec$ to $T_{\ell,0}$.
We identify $\spec$ as a fundamental domain in $T_{\ell,0}$, still denoted by $\spec$. Then for any $e^H \in T_{\ell,0}$, $\theta\in W$, we have $m_{\Pi}(e^H)=m_{\Pi}(e^{\theta(H)})$.

\begin{notation}
Furthermore, if $V$ and  $\Pi$ are $Z(\g)$ graded, then
for any $v \in  L^2(B)$, 
we have the decomposition 
$$\displaystyle v=\sum_{\kk \in Z(\g)} v_{\kk},$$
where $v_{\kk} \in V_k$, called a {\it $\kk$-graded vector}.
\end{notation}
\begin{definition}
Suppose $\Pi$ is a GUS-rep of $R_{\ell}$ over $\g$ at level $\ell$. 
We define a group action of $ T_{\sR}$ on $L^2(B)$: For any $e^{H'} \in T_{\sR}$ and any $\kk$-graded vector $v_{\kk}$,
$$e^{H'} \circ  v_{\kk} :=e^{i \langle \kk, H' \rangle} v_{\kk}.$$
\end{definition}

\begin{theorem}\label{Thm:mp=mp}
Suppose $\Pi$ is a GUS-rep of $R_{\ell}$ over $\g$ at level $\ell$. 
For any $e^H\in T_{\ell} $ and $e^{H'} \in T_{\sR}$, the following are equivalent:
\begin{itemize}
\item[(1)] $\Pi_{\ww}(v)=\chi_{\ww}(e^{H}) v$;
\item[(2)] $\Pi_{\ww}(e^{H'} \circ v)=\chi_{\ww}(e^{H-H'}) e^{H'} \circ v.$
\end{itemize}
Consequently,
$m_{\Pi}(e^H)=m_{\Pi}(e^{H-H'})$.
\end{theorem}

\begin{proof}
Suppose $\displaystyle v=\sum_{\kk \in Z(\g)} v_{\kk}$, where $v_{\kk}$ is graded by $\kk$.
If $\Pi_{\ww}(v)=\chi_{\ww}(e^{H}) v$,
then 
$$\chi_{\ww}(e^H)v=\Pi_{\ww} v=\sum_{\kk \in Z(\g)} \Pi_{\ww} v_{\kk}.$$ 
So $\Pi_{\ww} v_{\kk}=\chi_{\ww}(e^H) v_{\kk+\ww}$.
For any $e^{H'} \in T_{\sR}$,
$$\Pi_{\ww}(e^{H'} \circ v)=\sum_{\kk \in Z(\g)} e^{i \langle \kk, H' \rangle}  \Pi_{\ww}v_{\kk}=\sum_{\kk \in Z(\g)} e^{i \langle \kk, H' \rangle}\chi_{\ww}(e^H) v_{\kk+\ww}=\chi_{\ww}(e^H) e^{i \langle -\ww, H' \rangle} (e^{H'} \circ v).
$$
For any $\kk\in \sWl$, $\theta \in W$, we have $\theta(\kk)-\kk \in \sRl$.
For any $e^{H'} \in T_{\sR}$,  
$$\wchi_{\kk}(e^{H+H'})= \wchi_{\kk}(e^{H})e^{i \langle \kk, H' \rangle}.$$
So,
$$\chi_{\ww}(e^H) e^{i \langle \ww, H' \rangle}=\frac{\wchi_{\ww+\rho}(e^H) e^{i \langle \ww+\rho, H' \rangle}}{\wchi_{\rho}(e^H) e^{i \langle \rho, H' \rangle}} =\frac{\wchi_{\ww+\rho}(e^{H+H'})}{\wchi_{\rho}(e^{H+H'})}=\chi_{\ww}(e^{H+H'}).$$
Therefore $$\Pi_{\ww}(e^{H'} \circ v)=\chi_{\ww}(e^{H-H'}) e^{H'} \circ v.$$

\end{proof}

\section{Quantum Coxeter exponents}\label{Sec: quantum Coxeter exponents}

In this section, we generalize the root space, Coxeter elements and exponents for US-rep over $\g$ at level $\ell$, and answer the questions of Kac and Gannon.

Suppose $\Pi$ is a US-rep of $R_{\ell}$ over $\g$ at level $\ell$, and $B$ is a common eigenbasis of $\Pi_{\ww}$, $\ww \in \fw$.
We lift the spectrum of $\Pi$ from $\spec$ to $T_{\ell,0}$ and define the corresponding eigenspace:

\begin{definition}\label{Def: el}
We define the lifted eigenspace  $\El$ as
$$\El:= \{ \tg \in L^2(T_{\ell}) \otimes \hil : (\chi_{\ww} \otimes I-I\otimes \Pi_{\ww}) \tg=0, \; \forall \; \ww \in \fw \}.$$
\end{definition}

\begin{proposition}\label{prop: eigenbasis}
The set $\Bl:=\{ \delta_{e^H} \otimes v : e^H \in T_{\ell,0}, v\in B(e^H) \}$ is
an ONB of $\El$.
\end{proposition}

\begin{proof}

Note that the set
$$\{ \delta_{e^H} \otimes v : e^H \in T_{\ell}, v\in B \}$$ 
is an ONB of  $L^2(T_{\ell})\otimes \hil$, so for any $\tg \in L^2(T_{\ell}) \otimes \hil$, we have the decomposition
$$\tg=\sum_{e^H \in T_{\ell}} \sum_{v \in B} \lambda_{e^H,v} \delta_{e^H} \otimes v, $$
for some $ \lambda_{e^H,v} \in \mathbb{C}$.
If $\tg\in \El$, namely $(\chi_{\ww} \otimes I)\tg=(I\otimes \Pi_{\ww}) \tg,$ for any $\ww \in \fw$, then
$$\sum_{e^H \in T_{\ell}} \sum_{v \in B} \chi_{\ww}(e^H) \lambda_{e^H,v} \delta_{e^H} \otimes v=\sum_{e^H \in T_{\ell}} \sum_{v \in B} \chi_{\ww}(sp(v)) \lambda_{e^H,v} \delta_{e^H} \otimes v.$$
If $\lambda_{e^H,v} \neq 0$, then $ \chi_{\ww}(e^H)=\chi_{\ww}(sp(v))$, $\forall~ \ww\in \fw$.
By Lemma \ref{Lem: spec}, there is an $\theta \in W$, such that $sp(v)=e^{\theta(H)}$, equivalently $v \in B(e^H)$.

On the other hand, for any $v\in B(e^H)$, we have 
$$(\chi_{\ww} \otimes I -I\otimes \Pi_{\ww})(\delta_{e^H} \otimes v)=(\chi_{\ww}(e^H)-\chi_{\ww}(e^H) ) (\delta_{e^H} \otimes v)=0,$$
so $\delta_{e^H} \otimes v \in \El$.
Therefore $\Bl$ is an ONB of $\El$.
\end{proof}

\begin{definition}
For any $f\in L^2(\sWl)$, we define 
$$(\Delta_{\kk}f)(\jj)=\sum_{\ssb \in \mathcal{W}(\kk)} m_{\kk}(\ssb) f(\jj+\ssb), \forall j\in \sWl,$$
where $\mathcal{W}(\kk)$ is the weight diagram of $\kk \in \mathcal{C}$, and $m_{\kk}$ is the multiplicity function.
\end{definition}

\begin{lemma}\label{Lem: fd=wf}
Let $\mathcal{F}:L^2(\sWl) \to L^2(T_{\ell})$ be the Fourier transform in Equation \eqref{Equ:F}. Then for any $\ww\in \fw$, we have
\begin{align*}
\mathcal{F} \Delta_{\ww} = \chi_{\ww} \mathcal{F}.
\end{align*}
\end{lemma}

\begin{proof}
For any $\kk\in \sWl$,
\begin{align*}
\mathcal{F} \Delta_{\ww} (\delta_{\kk})
=\sum_{e^H \in T_{\ell}} \sum_{\ssb \in \mathcal{\ww}} m_{\ww}(\ssb) e^{i \langle \kk+\ssb, H \rangle}  \delta_{e^H}
=\sum_{e^H\in T_{\ell}} \chi_{\ww}(e^H) e^{i \langle \kk, H \rangle} \delta_{e^H}
=\chi_{\ww} \mathcal{F}(\delta_{\kk}).
\end{align*}

\end{proof}

\begin{definition}\label{Def: hl}
For a US-rep $\Pi$,
we define the full quantum root space $\Hl$ as
$$\Hl:=\{ \tf \in L^2(\sWl) \otimes \hil : (\Delta_{\ww} \otimes I-I\otimes \Pi_{\ww})\tf=0, \forall ~\ww \in \fw \}.$$
Furthermore, if $\Pi$ is $Z(\g)$-graded,
we define the quantum root space $\Hlz$ as
$$\Hlz:=\{ \tf \in \bigoplus_{\kk\in \sWl} \mathbb{C}\delta_{\kk} \otimes \hil_{-\kk} : (\Delta_{\ww} \otimes I-I\otimes \Pi_{\ww})\tf=0, \forall ~\ww \in \fw \}.$$ 
\end{definition}

\begin{remark}
When $G$ is an $ADE$ Dynkin diagram, $\Pi_G$ is a GUS-rep graded by $\Z_2=Z(\ssl_2)$, and $\Hlz$ is the space of additive functions on the corresponding Auslander-Reiten quiver, isomorphic to the root space. 
\end{remark}

\begin{notation}
Let $P_{\El}$ be the orthogonal projection from $L^2(T_{\ell}) \times \hil$ to $\El$.
Let $P_{\Hl}$ be the orthogonal projection from $L^2(\sWl) \otimes \hil$ to $\Hl$.
\end{notation}

\begin{proposition}\label{prop: fp=pf}
The Fourier transform $\mathcal{F} \otimes I$ is a unitary transformation from $\Hl$ to $\El$,
namely
$$(\mathcal{F} \otimes I) P_{\Hl}=P_{\El}(\mathcal{F} \otimes I).$$
\end{proposition}

\begin{proof}
It follows from Lemma \ref{Lem: fd=wf} and Definitions \ref{Def: el} and \ref{Def: hl}.
\end{proof}

Now we give the inner product formula on $\Hl$ using Fourier duality between $\Hl$ and $\El$.

\begin{theorem}\label{Thm:inner product}
For any $\kk,\jj \in \sWl$ and $v_1, v_2 \in \hil$,
we have 
\begin{equation}\label{Equ: inner}
\langle P_{\Hl} (\delta_{\kk}\otimes v_1),   P_{\Hl} (\delta_{\jj} \otimes v_2) \rangle
=\frac{1}{|T_{\ell}|} \sum_{\theta\in W} \varepsilon(\theta)  \langle v_1, \Pi(\wchi_{\jj-\kk+\theta(\rho)}) v_2 \rangle.
\end{equation}
\end{theorem}
\begin{proof}
For any $e^H \in \spec$, $B(e^H)$ is an eigenbasis of the eigenspace in $L^2(T_{\ell,0})^W$ with spectrum $e^H$. Then 
\begin{align*}
 P_{\El} (\mathcal{F}\delta_{\kk} \otimes v_1 )
=P_{\El} (\sum_{e^H \in T_{\ell}} e^{i\langle \kk, H \rangle} \delta_{e^H} \otimes v_1 ) 
=\sum_{e^H \in T_{\ell,0}}  \sum_{v \in B(e^H)} e^{i\langle \kk, H \rangle} \langle v, v_1 \rangle \delta_{e^H}\otimes v.
\end{align*}

By Proposition \ref{prop: fp=pf},
\begin{align*}
&\langle P_{\Hl} (\delta_{\kk} \otimes v_1) ,   P_{\Hl} (\delta_{\jj} \otimes v_2 )\rangle\\
=&\langle P_{\El} (\mathcal{F} \delta_{\kk} \otimes v_1) ,  P_{\El} (\mathcal{F} \delta_{\jj} \otimes v_2) \rangle\\
=&\frac{1}{|T_{\ell}|}\sum_{e^H \in T_{\ell,0}}  \sum_{v \in B(e^H)} \overline{e^{i\langle \kk, H \rangle} \langle v, v_1 \rangle} e^{i\langle \jj, H \rangle} \langle v, v_2 \rangle\\
=&\frac{1}{|T_{\ell}|}\sum_{e^H \in T_{\ell,0}}  \sum_{v \in B(e^H)}  e^{i\langle \jj-\kk, H \rangle} \langle v_1, v \rangle \langle v, v_2 \rangle \\
=&\frac{1}{|T_{\ell}|}\sum_{e^H \in \spec} \sum_{v \in B(e^H)}  \sum_{\theta'\in W}  e^{i\langle \jj-\kk, \theta'(H) \rangle}  \langle v_1, v\rangle \langle v, v_2 \rangle  \\
=&\frac{1}{|T_{\ell}|}\sum_{\theta\in W} \sum_{e^H \in \spec} \sum_{v \in B(e^H)}  \sum_{\theta'\in W} \varepsilon(\theta)  \frac{e^{i\langle \jj-\kk+\theta(\rho), \theta'(H) \rangle} }{\wchi_{\rho}(e^{\theta'(H)})} \langle v_1, v\rangle \langle v, v_2 \rangle \\
=&\frac{1}{|T_{\ell}|}\sum_{\theta\in W} \sum_{e^H \in \spec} \sum_{v \in B(e^H)}  \sum_{\theta'\in W} \varepsilon(\theta\theta')  \frac{e^{i\langle \theta'(\jj-\kk+\theta(\rho)), H \rangle} }{\wchi_{\rho}(e^H)} \langle v_1, v\rangle \langle v, v_2 \rangle \\
=&\frac{1}{|T_{\ell}|}\sum_{\theta\in W} \sum_{e^H \in \spec} \sum_{v \in B(e^H)}  \varepsilon(\theta)  \frac{\wchi_{\jj-\kk+\theta(\rho)}(e^H)}{\wchi_{\rho}(e^H)} \langle v_1, v\rangle \langle v, v_2 \rangle \\
=&\frac{1}{|T_{\ell}|}\sum_{\theta\in W} \sum_{e^H \in \spec} \sum_{v \in B(e^H)} \varepsilon(\theta)   \langle v_1, \Pi(\wchi_{\jj-\kk+\theta(\rho)})v\rangle \langle v, v_2 \rangle\\
=&\frac{1}{|T_{\ell}|} \sum_{\theta\in W} \varepsilon(\theta) \langle v_1,  \Pi(\wchi_{\jj-\kk+\theta(\rho)}) v_2 \rangle.
\end{align*}
\end{proof}

\begin{remark}
When $G$ is the quiver of a module of quantum $\ssl_n$ at level $\ell$, namely a higher Dynkin diagram,
$\Pi_G$ is a GUS-rep graded by $\Z_n=Z(\ssl_n)$.
Ocneanu called $\Hl$ the space of biharmonic functions and outlined a proof of the inner product formula \eqref{Equ: inner} in the course \cite{Ocn17}.
Our proof is different from the one outlined by Ocneanu.
\end{remark}

\begin{corollary}\label{Cor:Proj}
For any $\jj \in \sWl$ and $v \in \hil$, we have
$$P_{\Hl} (\delta_{\jj} \otimes v)=\frac{1}{|T_{\ell}|} \sum_{\kk \in \sWl}  \sum_{\theta\in W} \varepsilon(\theta)  \delta_{\kk}  \otimes \Pi(\wchi_{\jj-\kk+\theta(\rho)}) v .$$
\end{corollary}

\begin{proof}
It follows from checking the inner product with $\delta_{\kk}\otimes v_1$ on both sides using Theorem \ref{Thm:inner product}.
\end{proof}

\begin{remark}
If $B$ is a $\Z$-basis, then the functions $\{|T_{\ell}| P_{\Hl} (\delta_{\jj} \otimes v) \in \Hl: \jj\in \sWl, ~ v\in B\}$ have integer coefficients on the ONB $\{\delta_k : k \in \sWl\} \otimes B$. We consider them as a generalization of additive functions on the Auslander-Reiten quivers given in Equation \eqref{Equ: additive}.
\end{remark}

\begin{corollary}
When $\Pi$ is a GUS-rep,
let $P_{0}$ be the orthogonal projection from $L^2(\sWl) \times \hil$ onto the neutral subspace $\displaystyle \bigoplus_{\kk\in \sWl} \mathbb{C}\delta_{\kk} \otimes \hil_{-\kk}$. Then 
$$P_{\Hl} P_{0}=P_0 P_{\Hl},$$
and it is the projection onto $\Hlz$.
\end{corollary}
\begin{proof}
When $\Pi$ is a GUS-rep, 
by Corollary \ref{Cor:Proj}, for any $\jj \in \sWl$ and $v \in \hil_{-\jj}$, we have that
$$P_{\Hl} (\delta_{\jj} \otimes v)=P_0P_{\Hl} (\delta_{\jj} \otimes v).$$
Therefore, $P_{\Hl} P_0=P_0P_{\Hl}  P_0$. Both $P_{\Hl}$ and $P_0$ are projections, so 
$$P_{\Hl} P_{0}=P_0 P_{\Hl}.$$
By definition, it is the projection onto $\Hlz$. 
\end{proof}

Therefore we also call $\Hlz$ the neutral subspace of $\Hl$, as its total grading is 0 in $Z(\g)$. 

\begin{corollary}\label{Cor:length}
For any $\jj \in \sWl$ and $v \in \hil$, we have
$$\| P_{\Hl} (\delta_{\jj}\otimes v) \|_2^2=\frac{|W|}{|T_{\ell}|}\|v\|_2^2.$$
\end{corollary}
\begin{proof}
It follows from Theorem \ref{Thm:inner product} and that $\varepsilon(\theta) \Pi_{\theta(\rho)}=\Pi_{\rho}$ is the identity.
\end{proof}

\begin{definition}
We define the quantum root sphere as
$$\SPi:=\{\sqrt{|T_{\ell}|} P_{\Hl} (\delta_{\kk}\otimes v)  :  \kk\in \sWl, v \in \hil, \|v\|_2=1 \}.$$ 
For an ONB $B$ of $\hil$, we define the full quantum root system as  
$$\SB:=\{\sqrt{|T_{\ell}|} P_{\Hl} (\delta_{\kk}\otimes v)  :  \kk\in \sWl, v \in B, \|v\|_2=1 \}.$$
When $\Pi$ and $B$ are $\Z(\g)$-graded, we define the quantum root system as
$$\SBz:=\{\sqrt{|T_{\ell}|} P_{\Hl} (\delta_{\kk}\otimes v)  :  \kk\in \sWl, v \in B_{-\kk}, \|v\|_2=1 \}.$$
\end{definition}
By Corollary \ref{Cor:length}, any vector in $\SPi$ has length $\sqrt{|W|}$.
By Theorem \ref{Thm:inner product}, the inner product of vectors in $\SB$ is determined by the adjacent matrices.
Moreover, the (full) quantum root space is spanned by the (full) quantum root system. 

\begin{remark}
When $G$ is an $ADE$ Dynkin diagram, $\SBz$ is a standard realization of the root system by additive functions, see \cite{Gab80,Hap87}
\end{remark}

For any $\jj\in \sWl$, the translation $\vartheta_\jj: \kk\mapsto \kk+\jj$ on $\sWl$ induces a dual action on $L^2(\sWl)$, still denoted by $\vartheta_{\jj}$:
$$\vartheta_{\jj} (\delta_{\kk})=\delta_{\kk-\jj},  \; \forall \; \kk\in \sWl.$$
We define the corresponding translation on $ L^2(\sWl) \otimes \hil$ as
$$\ttj:= \vartheta_{\jj} \otimes I.$$ 
Then $\ttj$ is well-defined on $\Hl$ as well. 
Moreover, $\{\ttj\}_{\jj\in \sWl}$ is a finite abelian group.

\begin{proposition}
The quantum root sphere $\SPi$ and quantum root system $\SB$ are translation invariant by $\ttj$, $\jj \in \sWl$..
 
\end{proposition}
\begin{proof}
\end{proof}

\begin{theorem}\label{Thm: common eigenbasis}
The set $\Bl=\{ \delta_{e^H} \otimes v : e^H \in T_{\ell,0}, v\in B(e^H) \}$  is a common eigenbasis of  $\{\mathcal{F} \vartheta_\jj \mathcal{F}^{-1} \otimes I \}_{\jj\in \sWl}$ and $\{I \otimes \Pi_{\ww}\}_{\ww \in \fw}$ in $\El$.
\end{theorem}

\begin{proof}
By Proposition \ref{prop: eigenbasis}, $\Bl$ is a basis of $\El$.
By Fourier duality, 
$$\vartheta_\jj \mathcal{F}^{-1}\delta_{e^H}= e^{i \langle -\jj, H \rangle} \mathcal{F}^{-1} \delta_{e^H}, ~ \forall \jj \in \sWl,~ e^H \in T_{\ell,0}.$$ 
So
$$( \mathcal{F} \vartheta_\jj \mathcal{F}^{-1}\otimes I )(\delta_{e^H} \otimes v)= e^{i \langle -\jj, H \rangle} \delta_{e^H} \otimes v .$$
On the other hand, 
$$(I \otimes \Pi_{\ww} )(\delta_{e^H} \otimes v) = \chi_{\ww}(e^H)\delta_{e^H} \otimes v.$$
Therefore $\Bl$ is a common eigenbasis.
\end{proof}

\begin{definition}
By Theorem \ref{Thm: common eigenbasis}, if $\tf$ is a common eigenvector of the translations, then 
$$\ttj \tf = e^{ i \langle -\jj, H \rangle} \tf, ~ \forall ~\jj \in \sWl,$$
for some $e^H \in T_{\ell,0}$.
We call $e^H$ the spectrum of $\tf$, denoted by $sp_{\vartheta}(\tf)=e^H$.
\end{definition}

\begin{definition}
Suppose $\Pi$ is a US-rep and $A$ is a subgroup of $\sWl$. 
We define the multiplicity $m_{A}$ of $e^H$ to be dimension of the common eigenspace in $\Hl$ of the translations $\{\vartheta_\jj \}_{\jj \in A}$ with spectrum $e^H$, namely
$$m_{A}(e^H):= \dim \left\{ \tf \in \Hl : \vartheta_\jj \tf = e^{ i \langle -\jj, H \rangle} \tf, ~ \forall ~\jj \in A \right\}.$$
\end{definition}

For a subgroup $A$ of $\sWl=\sW/\sRl$, we consider it as an intermediate subgroup of $\sRl \subset \sW$. 
By Definition \ref{Def:TA},
\begin{align*}
\mathfrak{t}_{A}&=\{H\in \mathfrak{t} :\langle  \rr, H \rangle\in \mathbb{Z}, \; \forall \; \rr\in A \},\\
T_{A}&=\left\{e^{2\pi H} \in T: H \in \mathfrak{t}_{A} \right\} \subseteq T_{\ell}.
\end{align*}

\begin{theorem}\label{Thm:m=m}
Suppose $\Pi$ is a US-rep and $A$ is a subgroup of $\sWl$. 
For any $e^H\in T_{\ell,0}$,
\begin{equation}\label{Equ:m=m}
m_{A}(e^H)= \sum_{e^{H'} \in T_{A}} m_{\Pi}(e^{H+H'}).
\end{equation}
In particular,
$$m_{\sW}(e^H)=m_{\Pi}(e^H)=m_{\sW}(e^{\theta(H)}), \; \forall \; \theta \in W.$$
\end{theorem}

\begin{proof}
Note that for any $e^{H_1} \in T_{\ell,0}$ and $v \in B(e^{H_1})$,
the following are equivalent:
\begin{itemize}
\item[(1)]
$\ttj (\delta_{e^{H_1}} \otimes v) = e^{ i \langle -\jj, H \rangle}\delta_{e^{H_1} }\otimes v, \; \forall \; \jj \in A;$
\item[(2)]
$\ttj (e^{H_1})=\ttj (e^{H}), \; \forall \; \jj \in A;$
\item[(3)]
$\ttj (e^{H_1-H})=1, \; \forall \; \jj \in A;$
\item[(4)]
$H_1-H \in T_A,$
\end{itemize}
where (3) $\iff$ (4) follows from Theorem \ref{Thm: group duality}.
By Theorem \ref{Thm: common eigenbasis}, we obtain Equation \eqref{Equ:m=m}.
Furthermore, if $A=\sW$, then $T_{\sW}$ is the trivial group. So
$$m_{\sW}(e^H)=m_{\Pi}(e^H)=m_{\Pi}(e^{\theta(H)})=m_{\sW}(e^{\theta(H)}), \; \forall \; \theta \in W.$$ 

\end{proof}

If $\Pi$ is a GUS-rep, then $\Hlz$ is the neutral subspace of $\Hl$, and $\Hlz$ is translation invariant under the action of $\ttj$, $\jj \in \sR$. 

\begin{definition}
Suppose $\Pi$ is a GUS-rep, and $A$ is a subgroup of $\sR/\sRl$. 
We define
$$m_{A,0}(e^H)= \dim \{ \tf \in \Hlz : \vartheta_\jj \tf = e^{ i \langle -\jj, H \rangle} \tf, ~ \forall ~\jj \in A \}.$$
\end{definition}

\begin{corollary}
Moreover, $m_{A}(e^H)=m_{A}(e^{H+H'})$ for any $e^{H'} \in T_{\sRa}$. 
\end{corollary}

\begin{proof}
It follows from Theorems \ref{Thm:m=m} and \ref{Thm:mp=mp}.
\end{proof}

\begin{theorem}\label{Thm:r}
Suppose $\Pi$ is a GUS-rep, and $A$ is a subgroup of $\sR/\sRl$. 
For any $e^H \in T_{\ell,0}$, and $e^{H'} \in T_{\sR}$, 
$$m_{A}(e^H)= m_{A}(e^{H+H'})= n_z m_{A,0}(e^H).$$
In particular,
$$m_{\sR,0}(e^H)= m_{\Pi}(e^{H}).$$
\end{theorem}

\begin{proof}
By Theorem \ref{Thm:mp=mp} and Theorem \ref{Thm:m=m}, we have that
\begin{align*}
m_{\Pi}(e^{H})&=m_{\Pi}(e^{H+H'});\\
m_{A}(e^H)&= m_{A}(e^{H+H'}).
\end{align*}
Note that $\displaystyle \Hl=\bigsqcup_{\kk \in Z(\g)} \vartheta_{\kk} \Hlz$,  so 
$$m_{A}(e^{H})= n_z m_{A,0}(e^{H}).$$ 
When $A=\sR$, by Theorem \ref{Thm:m=m},
$$m_{\sR}(e^H)= \sum_{e^{H'} \in T_{\sR}} m_{\Pi}(e^{H+H'})=|T_{\sR}| m_{\Pi}(e^{H})=n_z m_{\Pi}(e^{H}).$$
Therefore,
$$m_{\sR,0}(e^H)= m_{\Pi}(e^{H}).$$
\end{proof}

Recall that $\rooth$ is the set of highest positive roots of $\g$.
For any $\hr \in \rooth$, its order in the group $\sWl$ is the quantum Coxeter number $n_c$.
So the translation $\ttr$ on $\Hl$ (or $\Hlz$) has periodicity $n_c$, 

\begin{definition}
For any $\hr \in \rooth$, we call the translation $\ttr$ on $\Hl$ a quantum Coxeter element.  
\end{definition}

\begin{definition}
We define the exponent map $\Phi: T_{\ell,0} \times \rooth \to \Z_{n_c}$ as 
$$\Phi(e^H,\hr)=\frac{\langle \hr, H \rangle}{2\pi}.$$
Equivalently,
$$e^{\frac{2\pi i}{n_c} \Phi(e^H,\hr)}= e^{i \langle \hr, H \rangle }.$$
We call $\Phi(e^H, \hr)$ the quantum Coxeter exponent of the quantum Coxeter element $\hr$ at spectrum $e^H \in T_{\ell,0}$.
We call the map $\Phi(e^H,\cdot): \rooth \to \Z_{n_c}$ the quantum Coxeter exponent at the spectrum $e^H$.

\end{definition}
Since $\rooth \subset \sR$, $\SBz$ is invariant under the action of quantum Coxeter elements $\{\vartheta_{\theta(\rr+)}  \}_{\theta \in W}$. 
For the $\ssl_2$ case, $\vartheta_{-\rr_+}$ is also a Coxeter element on the root system. This induces a $\Z_2$ symmetry on the eigenvalue of the Coxeter element. In general, the eigenvalue of the quantum Coxeter elements has a Weyl group $W$ symmetry.

\begin{definition}
Suppose $\Pi$ is a US-rep.
For any $e^H\in T_{\ell,0}$, we define the multiplicity of the quantum Coxeter exponent $\Phi(e^H,\cdot)$ for the action of quantum Coxeter elements on $\Hl$ as 
$$m_{\Phi}(e^H):=\dim \left\{ \tf \in \Hl : \ttr \tf=e^{i  \langle \hr, H \rangle} \tf, \; \forall \; \hr \in \rooth \right\}.$$
When $\Pi$ is a GUS-rep, we define the multiplicity for the action on $\Hlz$ as
$$m_{\Phi,0}(e^H):=\dim \left\{ \tf \in \Hlz : \ttr \tf=e^{i  \langle \hr, H \rangle} \tf, \; \forall \; \hr \in \rooth \right\}.$$
\end{definition}

\begin{theorem}\label{thm: exponents}
Suppose $\Pi$ is a US-rep over quantum $\g$ at level $\ell$. For any $e^H \in T_{\ell,0}$,
$$m_{\Phi}(e^H)=m_{\sRa}(e^H)=\sum_{e^{H'} \in T_{\sRa}} m_{\Pi}(e^{H+H'}).$$
When $\Pi$ is a GUS-rep, 
$$m_{\Phi,0}(e^H)=m_{\sRa,0}(e^H)=\sum_{e^{H'} \in T_{\sRa}/T_{\sR}} m_{\Pi}(e^{H+H'}).$$
In particular, if $\g$ is an ADE Lie algebra, then $T_{\sR}=T_{\sRa}$. We have that 
$$m_{\Phi}(e^H)=m_{\Pi}(e^{H}).$$
\end{theorem}

\begin{proof}
Since $\sRa$ is generated by $\{\vartheta_{\theta(\rr+)}  \}_{\theta \in W}$,
for any $e^H \in T_{\ell}$, $m_{\Phi}(e^H)=m_{\sRa,0}(e^H)$.
By Theorem \ref{Thm:r}, $\displaystyle m_{\sRa,0}(e^H)=\frac{1}{n_z}m_{\sRa}(e^H)$.
By Theorems \ref{Thm:m=m}, $\displaystyle m_{\sRa}(e^H)=\sum_{e^{H'} \in T_{\sRa}} m_{\Pi}(e^{H+H'})$. By Theorem \ref{Thm:mp=mp}, $\displaystyle \frac{1}{n_z} \sum_{e^{H'} \in T_{\sRa}} m_{\Pi}(e^{H+H'})=\sum_{e^{H'} \in T_{\sRa}/T_{\sR}} m_{\Pi}(e^{H+H'})$.
\end{proof}

When $\g=\ssl_2$ and $G$ is an $ADE$ Dynkin diagram, we have the following correspondence:
\begin{enumerate}
\item $n_c$ is the Coxeter number;
\item $\Pi$ is a GUS-rep; 
\item  $\Pi_{\ww}$ is the adjacent matrix of $G$, where $\ww$ is the fundamental weight; 
\item $T_{\ell,0}=\left\{ e^H=\left[\begin{array}{cc} e^{\frac{k \pi i}{2n_c}} &0 \\ 0 & e^{\frac{-k \pi i}{2n_c}} \end{array}\right] : k=1, 2, \cdots, n_c-1, n_c+1, n_c+2, \cdots,2n_c-1 \right\};$ 
\item $\spec=\left\{ \left(\left[\begin{array}{cc} e^{\frac{k \pi i}{2n_c}} &0 \\ 0 & e^{\frac{-k \pi i}{2n_c}} \end{array}\right], \left[\begin{array}{cc} e^{\frac{-k \pi i}{2n_c}} &0 \\ 0 & e^{\frac{k \pi i}{2n_c}} \end{array}\right]\right) : k=1, 2, \cdots, n_c-1\right\};$ 
\item The eigenvalue of the adjacent matrix $\Pi_{\ww}$ at $e^H$ is $\chi_{\ww}(e^H)=e^{\frac{k \pi i}{2n_c}} + e^{\frac{-k \pi i}{2n_c}}$ with multiplicity $m_{\Pi}(e^H)$;
\item $\SBz$ is the root system; 
\item $\Hlz$ is the root space; 
\item $\{\ttr : \hr \in \rooth \}=\{\pm \rr_+\}$ is an opposite pair of Coxeter elements; 
\item The eigenvalue of the Coxeter element $\tilde{\vartheta}_{\pm\rr_+}$ at $e^H$ is $\chi_{\ww}(e^H)=e^{\frac{\pm k \pi i}{n_c}}$ with multiplicity $m_{\sRa,0}(e^H)$;
\item $\Phi(\pm \rr_+)=\pm k$, where $k$ is the Coxeter exponent, with multiplicity $m_{\Phi,0}(e^H)$.
\end{enumerate}

The classical correspondence of Coxeter exponents in the $ADE$ Lie theory is given by 
$$m_{\Pi}(e^H)=m_{\sRa,0}(e^H)=m_{\Phi,0}(e^H).$$
Theorem \ref{thm: exponents} is a generalization of this correspondence for any GUS-rep of the Verlinde algebra of any $ADE$ Lie algebra $\g$ at any level $\ell$.
A modified correspondence for any US-rep or GUS-rep of the Verlinde algebra of any simple Lie algebra $\g$ at any level $\ell$ is proved in Theorems \ref{Thm:m=m} and \ref{Thm:r} respectively.
We answer the question posed by Kac at MIT in 1994 and the recent comment posed by Gannon at Vanderbilt University in 2017 positively.

When $\g=\ssl_n$, we have $\sR=\sRa$, and 
$$ T_{\ell,0}=\left\{ diag(p_1,p_2,\ldots,p_n) : \Pi_{k=1}^n p_k=1,~ p_k\neq p_j, ~p_k^{n_c}=p_j^{n_c} \right\}.$$
Let $v_1,v_2,\cdots,v_n$ be the weights corresponding to the ONB of the standard representation of $\ssl_n$. Then 
$\displaystyle \sum_{k=1}^n v_k=0$.
For any $e^H \in T_{\ell}$, $e^{i \langle v_j, H \rangle }=p_j$.
Moreover, for any $\theta\in W \cong S_n$, its action on $\{v_1,v_2,\cdots,v_n\}$ is a permutation: $\theta(v_i)=v_{\theta(i)}$.
Note that $\rr_+=v_1-v_n$
The quantum Coxeter exponent is given by $\Phi(e^H,\theta(\rr_+))=\frac{\langle \theta(\rr_+), H \rangle}{2\pi}$, and
$$\displaystyle e^{\frac{ 2\pi i}{n_c} \Phi(e^H,\theta) } = e^{i \langle \hr, H \rangle }=\frac{p_{\theta(1)}}{p_{\theta(n)}}.$$

\section{Summary}\label{Sec: Summary}

\textbf{Main results:}
Let $\g$ be a simple complex Lie algebra such that it is the complexification of the Lie algebra $\mathfrak{k}$ of a simply-connected compact Lie group $K$. Let $\mathfrak{t}$ be a Cartan subalgebra and $T=\{e^H | H\in \mathfrak{t} \}$ be the maximal torus in $K$.
Let $c_\g^\vee$ be the dual Coxeter number, $\fw$ be the fundamental weights, $\rho$ be the Weyl vector, $\sW$ be the weight lattice and $\sR$ be the root lattice, $W$ be the Weyl group, and $\sRa$ be the sublattice of $\sR$ generated by $\rooth:=\{ \theta(\rr_+) | \theta \in W\}$, where $\rr_+$ is a highest positive root. Take $Z(\g)=\sW/\sR$ and $n_z=|Z(\g)|$. The exponential map $H\to e^{2 \pi H}$ is a group isomorphism 
$\mathfrak{t}/\mathfrak{R}\cong T$, where $\mathfrak{R}$ is the coroot lattice. Therefore the map $\hat{} : \kk\to \hat{\kk}$ given by
$$\hat{\kk}(e^H)=e^{2 \pi i \langle \kk, H \rangle},  ~\kk\in \sW,$$
is well defined on $e^H \in T$. 
The Fourier transform $~\hat{}~$ is a group isomorphism from $\sW$ to the dual of the abelian group $T$.

For any level $\ell \in \mathbb{N}$, take the quantum Coxeter number $n_c=c_\g^\vee+\ell$. We call $\sWl=\sW/n_c\sRa$ the weight torus.
We define its Fourier dual $T_{\ell}$ as a finite subgroup of $T$, which admits a Weyl group $W$ action. Let $\sW_{\ell,0}$ and $T_{\ell,0}$ be the corresponding subsets of elements off the Weyl mirrors. Take an alcove $\CL$ containing $\fw$ and $\spec$ to be the fundamental domains of $\sW_{\ell}$ and $T_{\ell,0}$ subject to the Weyl group action respectively.

Let $L^2(T_{\ell,0})^W$ be the Hilbert space of anti-symmetric functions with respect to the Weyl group action and the measure is a Haar measure.  We introduce a multiplication $\star$ and an involution $*$ on $L^2(T_{\ell,0})^W$ and show that $L^2(T_{\ell,0})^W$ becomes an abelian $C^*$ algebra.  
We construct an orthonormal basis (ONB) $\{|W|^{-\frac{1}{2}} \wchi_{\kk}, \kk \in \CL\}$ of $L^2(T_{\ell,0})^W$, and prove that $\{\wchi_{\kk}, \kk \in \CL\}$ form a $Z(\g)$-graded fusion ring $R_{\ell}$, and the corresponding $C^*$ is isomorphic to the Verlinde algebra  of quantum $\g$ at level $\ell$. Therefore, we call the anti-symmetric function $\wchi_{\kk}$ an $\ell$-character, and $L^2(T_{\ell,0})^W$ the $\ell$-character Verlinde algebra. In particular, the fusion coefficients of representations can be computed from the inner product of $\ell$-characters: $$\tilde{N}_{\kk,\jj}^{\ssb}=\frac{1}{|W|}\langle \wchi_{\kk} \star \wchi_{\jj}, \wchi_{\ssb}\rangle.$$

Suppose $\hil$ is a finite dimensional Hilbert space and $\Pi: R_{\ell} \to \hom(\hil)$ is a unital, *-representation. 
We prove that for any common eigenvector $v \in \hil$ of the matrices $\{ \Pi_{\ww}:=\Pi(\wchi_{\ww+\rho}) \}_{\ww \in \fw}$, there is a $e^H \in \spec$, such that 
$$\Pi_{\ww} v=\chi_{\ww}(e^H)v, \ww\in \fw,$$
where $\chi_{\ww}$ is the Weyl character.  Therefore we call $\spec$ the spectrum of $R_{\ell}$ and $e^H$ the spectrum of $v$. 
We define the multiplicity $m_{\Pi}$ of $e^H$ to be the dimension of the corresponding eigenspace in $\hil$.

Furthermore, we lift the spectrum from $\spec$ to $T_{\ell,0}$ and define the corresponding eigenspace
$$\El:= \{ \tg \in L^2(T_{\ell}) \otimes \hil \; | \; (\chi_{\ww} \otimes I-I\otimes \Pi_{\ww}) \tg=0, \; \forall \; \ww \in \fw \}.$$
The Fourier dual of $\El$ is 
$$\Hl:=\{ \tf \in L^2(\sWl) \otimes \hil \; | \; (\Delta_{\ww} \otimes I-I\otimes \Pi_{\ww}) \tf=0, \; \forall \; \ww \in \fw \},$$
generalizing additive functions on the Auslander-Reiten quiver. We call $\Hl$ the full quantum root space.  
Take the orthogonal projection $P_{\Hl}: L^2(\sWl) \otimes \hil \to \Hl$. We prove the following inner product formula in Theorem \ref{Thm:inner product}:
For any $\kk,\jj \in \sWl$ and $v_1, v_2 \in \hil$,
we have 
\begin{equation*}
\langle P_{\Hl} (\delta_{\kk}\otimes v_1),   P_{\Hl} (\delta_{\jj} \otimes v_2) \rangle=\frac{1}{|T_{\ell}|} \sum_{\theta\in W} \varepsilon(\theta)  \langle v_1, \Pi(\wchi_{\jj-\kk+\theta(\rho)}) v_2 \rangle.
\end{equation*}
Based on the inner product formula, we define the quantum root sphere and quantum root system.

Let $\vartheta_\jj$, $\jj \in \sWl$ be the translation by $\jj$ on $\sWl$ and the induced action on $\Hl$ is denoted by $\ttj$.
We prove that if $\tf$ is a common eigenvector of the translations, then 
$$\ttj \tf = e^{ i \langle -\jj, H \rangle} \tf, ~ \forall ~\jj \in \sWl$$
for some $e^H \in T_{\ell,0}$.
We call $e^H$ the spectrum of $\tf$.

For any intermediate group $A$ of $n_c\sRa \subset \sW$,
we define the Fourier dual of $\sW/A$ as $$T_{A}=\left\{e^{2\pi H} \in T | H \in \mathfrak{t}_{A} \right\} \subset T_{\ell},$$
where $\mathfrak{t}_{A}=\{H\in \mathfrak{t} | \langle  \rr, H \rangle\in \mathbb{Z}, \forall \rr\in A \}$.
We define the multiplicity $m_{A}$ of $e^H$ to be the dimension of the common eigenspace in $\Hl$ of the translations $\{\ttj\}_{\jj \in A}$ with spectrum $e^H$. We prove our main theorem that, for any $e^H\in T_{\ell,0}$,
$$m_{A}(e^H)= \sum_{e^{H'} \in T_{A}} m_{\Pi}(e^{H+H'}).$$
In particular, $m_{\sW}(e^H)= m_{\Pi}(e^{H})$.

For any $\hr \in \rooth$, the order of $\hr$ in the group $\sWl$ is the quantum Coxeter number $n_c$.
So the translation $\vartheta_{\theta(\rr+)}$ has periodicity $n_c$, and we call it a quantum Coxeter element.  
We define the exponent map $\Phi: T_{\ell,0} \times \rooth \to \Z_{n_c}$, such that 
$$e^{\frac{2\pi i}{n_c} \Phi(e^H,\hr)}= e^{i \langle \hr, H \rangle }.$$
We define the multiplicity of the quantum Coxeter exponent $\Phi(e^H,\cdot)$ for the action of quantum Coxeter elements on $\Hl$ as 
$$m_{\Phi}(e^H)=\dim \{ f \in \Hl \; | \;  \ttr f=e^{i  \langle \hr, H \rangle} f, \; \forall \; \hr \in \rooth\}.$$
Then
$$m_{\Phi}(e^H)=m_{\sRa}(e^H)=\sum_{e^{H'} \in T_{\sRa}} m_{\Pi}(e^{H+H'}).$$

Furthermore, if the representation $\Pi$ is $Z(\g)$-graded, then the full quantum root space $\Hl$ is also $Z(\g)$-graded. Moreover, $\Hl$ decomposes into $n_z$ copies of $\Hlz$, called the quantum root space, where $\Hlz$ is the neutral subspace and it is invariant under the action of $\vartheta_{\jj}$, for any $\jj \in \sR$. 
In particular, $\Hlz$ is invariant under the action of quantum Coxeter elements $\{\vartheta_{\hr}  \}_{\hr \in \rooth}$. 

For an intermediate subgroup $A$ of $n_c\sRa \subset \sR$, we define the multiplicity $m_{A,0}$ of $e^H$ to be the dimension of the common eigenspace in $\Hlz$ of the translations $\{\ttj\}_{\jj \in A}$ with spectrum $e^H$. Then $m_{A}=n_z m_{A,0}$, and $m_{A}(e^H)=m_{A}(e^{H+H'})$ for any $e^{H'} \in T_{\sRa}$. In particular, 
$$m_{\sR,0}(e^H)=m_{\Pi}(e^{H}).$$ 
We define the multiplicity of the quantum Coxeter exponent $\Phi(e^H,\cdot)$ for the action of quantum Coxeter elements on $\Hlz$ as 
$$m_{\Phi}(e^H)=\dim \{ f \in \Hl \; | \;  \ttr f=e^{i  \langle \hr, H \rangle} f, \; \forall \; \hr \in \rooth\},$$
and prove that 
$$m_{\Phi,0}(e^H)=m_{\sRa,0}(e^H)=\sum_{e^{H'} \in T_{\sRa}/T_{\sR}} m_{\Pi}(e^{H+H'}).$$
In particular, if $\g$ is an ADE Lie algebra, then $T_{\sR}=T_{\sRa}$. We have that 
$$m_{\Phi,0}(e^H)=m_{\Pi}(e^{H}).$$
This generalizes the correspondence of Coxeter exponents in the $ADE$ Lie theory and we answer the questions of Kac and Gannon.

\end{document}